\documentclass[twoside,12pt]{article}
\usepackage{geometry}
\usepackage{graphicx,graphics}
\usepackage{fancyhdr}
\usepackage{amssymb,amsmath,mathrsfs,amsfonts,amsthm}
\usepackage{latexsym}
\usepackage{array}
\usepackage{bm}
\usepackage{subfig}
\usepackage{paralist}
\usepackage[colorlinks, linkcolor=blue,  citecolor=blue,urlcolor=blue]{hyperref}
\usepackage[dvips]{epsfig}
\usepackage{subfig}
\usepackage{cite}

\theoremstyle{plain}                    
\newtheorem{thm}{Theorem}[section]

\theoremstyle{definition}               

\theoremstyle{definition}
\newtheorem{example}{Example}

\theoremstyle{remark}
\newtheorem{rmk}[thm]{Remark}

\newenvironment{acknowledgments}{{\flushleft \bf Acknowledgment:}}{}

\numberwithin{equation}{section}
\numberwithin{figure}{section}
\numberwithin{table}{section}

\allowdisplaybreaks[1]

\numberwithin{equation}{section}
\numberwithin{figure}{section}
\numberwithin{table}{section}

\title{A Finite-Volume Method for Nonlinear Nonlocal Equations with a Gradient Flow Structure}
\author{Jos\'e A. Carrillo\thanks{Department of Mathematics, Imperial College London, London SW7 2AZ, UK; {\tt carrillo@imperial.ac.uk}},
{} Alina Chertock\thanks{Department of Mathematics, North Carolina State University, Raleigh, NC 27695, USA;
{\tt chertock@math.ncsu.edu}},
~{}and Yanghong Huang\thanks{Department of Mathematics, Imperial College London, London SW7 2AZ, UK; {\tt 
yanghong.huang@imperial.ac.uk}}}
\date\today

\begin{document}

\maketitle

\begin{abstract}
We propose a positivity preserving entropy decreasing finite volume scheme for nonlinear nonlocal equations with a gradient flow structure. These properties allow for accurate computations of stationary states and long-time asymptotics demonstrated by suitably chosen test cases in which these features of the scheme are essential. The proposed scheme is able to cope with non-smooth stationary states, different time scales including metastability, as well as concentrations and self-similar behavior induced by singular nonlocal kernels. We use the scheme to explore properties of these equations beyond their present theoretical knowledge.
\end{abstract}

\section{Introduction}\label{sec1}
In this paper, we consider a finite-volume method for the following problem:
\begin{equation}
\left\{\begin{aligned}
&\rho_t=\nabla\cdot\big[\rho\nabla\big(H'(\rho)+V(\mathbf{x})+
 W\ast\rho\big)\big],
\quad \mathbf{x}\in\mathbb{R}^d,\ t>0,\\
&\rho(\mathbf{x},0)=\rho_0(\mathbf{x}),
\end{aligned}\right.
\label{1.1}
\end{equation}
where $\rho(\mathbf{x},t)\ge0$ is the unknown probability measure, $W(\mathbf{x})$ is an interaction potential, which is assumed to be symmetric, $H(\rho)$ is a density of  internal energy, and $V(\mathbf{x})$ is a confinement potential.

Equations such as \eqref{1.1} appear in various contexts. If $W$ and $V$  vanishes, and 
$H(\rho)=\rho\log \rho-\rho$ or $H(\rho) =\rho^{m}$, it is the classical heat equation or porous medium/fast diffusion equation 
\cite{Vaz}. If mass-conserving,  self-similar solutions of these equations are sought, the quadratic term 
$V(\mathbf{x})=|\mathbf{x}|^2$ is added, leading to new equations in  similarity variables. More generally, $V$ usually appears as a confining potential in Fokker-Planck type equations \cite{CT00,MR1842429}.  
Finally, $W$ is related to the interaction energy, 
and can be as singular as the Newtonian  potential in chemotaxis 
system~\cite{keller1970initiation} or as smooth as $W(\mathbf{x})=|\mathbf{x}|^\alpha$ 
with $\alpha>2$ in granular flow~\cite{MR1471181}.

The free energy associated to equation \eqref{1.1} is given by (see \cite{cmcv-03,cmcv-06,MR1964483}):
\begin{equation}
E(\rho)=\int_{\mathbb{R}^d}H(\rho)\,d\mathbf{x}+
\int_{\mathbb{R}^d}V(\mathbf{x})\rho(\mathbf{x})\,d\mathbf{x}
+\frac{1}{2}\int_{\mathbb{R}^d}\int_{\mathbb{R}^d}W(\mathbf{x}-\mathbf{y})\rho(\mathbf{x})\rho(\mathbf{y})
\,d\mathbf{x}\,d\mathbf{y}\,.
\label{fe}
\end{equation}
This energy functional is the sum of  internal energy, potential energy
and interaction energy, corresponding to the three 
terms on the right-hand side of \eqref{fe}, respectively. A simple computation 
shows that, at least for classical solutions, the time-derivative 
of $E(\rho)$ along solutions  of~\eqref{1.1} is 
\begin{equation}
\frac{d}{dt}E(\rho)=-\int_{\mathbb{R}^d}\rho|\mathbf{u}|^2\,d\mathbf{x}
:=-I(\rho),
\label{ed}
\end{equation}
where
\begin{equation}
\mathbf{u}=-\nabla\xi,\quad \xi:=\frac{\delta E}{\delta \rho}=
H'(\rho)+V(\mathbf{x})+W\ast\rho.
\label{xi}
\end{equation}
The functional $I$ will henceforth be referred to as the entropy dissipation functional. 

The equation~\eqref{1.1} and its associated energy $E(\rho)$ are the subjects
of intensive study during the past fifteen years, see e.g. \cite{MR1451422,MR2401600,MR1964483,cmcv-03} and the references therein.
The general properties of~\eqref{1.1} are investigated in the context of 
interacting gases~\cite{MR1451422,MR1964483,cmcv-03}, and  are common to a wide variety of models, including granular flows 
\cite{MR1471181,BCCP,MR1812737,LT04}, porous medium flows \cite{CT00,MR1842429}, and collective behavior in biology \cite{MR2257718}. The gradient flow structure, in the sense of~\eqref{ed}, is generalized from smooth solutions to measure-valued solutions \cite{MR2401600}. Certain entropy-entropy dissipation inequalities between $E(\rho)$ and  $I(\rho)$ are also recognized to characterize the fine details of the convergence to steady states \cite{CT00,MR1842429,cmcv-03}. 

The steady state of~\eqref{1.1}, if it exists, usually verifies the form
\begin{equation}\label{eq:steadyeq}
\xi = H'(\rho)+V(\mathbf{x})+W\ast\rho = C,\quad  \mbox{on supp } \rho,
\end{equation}
where the constant $C$ could be different on different connected components  of $\mbox{supp } \rho$. In many cases, especially in the presence of the interaction potential $W$, there are multiple steady states, whose explicit forms are available only for particular $W$. Most of studies of these steady states are based on  certain assumptions on the support and the characterizing equation \eqref{eq:steadyeq}.

In this work, we propose a positivity preserving finite-volume method to treat the general nonlocal nonlinear PDE \eqref{1.1}. Moreover, we show the existence of a discrete free energy that is dissipated for the semi-discrete scheme (discrete in space only). A related method was already proposed in \cite{Filbet} for the case of nonlinear degenerate diffusions in any dimension. We generalize this method to cover the nonlocal terms for both 1D and 2D cases in Section \ref{sec2}. In fact, the first order scheme generalizes easily to cover unstructured meshes. However, it is an open problem how to obtain entropy decreasing higher order schemes in this setting in 2D. Let us remark that other numerical methods based on finite element approximations have been proposed in the literature which are positivity preserving and entropy decreasing at the expense of constructing them by an implicit discretization in time but continuous in space, see \cite{BCW}.

Section \ref{sec3} is devoted to numerical experiments, in which the performance of the developed numerical approach is tested. In Section~\ref{sec31}, we conduct the convergence study of stationary states, where the order of accuracy depends on the regularity at free boundaries. We then showcase the performance of this method for finding stable stationary states with nonlocal terms and their 
equilibration rate in time for different nonlocal models. In Section \ref{sec33}, we emphasize how this method is useful to explore different open problems in the analysis of these nonlocal nonlinear models such as the Keller-Segel model for  chemotaxis in its different  versions. We continue in Section \ref{sec:attrep} with aggregation equations with repulsive-attractive kernels and address the issue of singular kernels and discontinuous steady states. Finally, in Section \ref{sec34}, we demonstrate the performance of the scheme in a number of 2-D experiments showcasing numerical difficulties and interesting asymptotics. 


\section{Numerical Method}\label{sec2}

In this section, we describe both one- (1-D) and two-dimensional (2-D) finite-volume schemes for \eqref{1.1} and prove their positivity 
preserving and entropy dissipation properties. We also establish error estimates and convergence results for the proposed methods. We start in  \S\ref{sec21} with the 1-D case and then generalize it to the 2-D case in \S\ref{sec22}, both on uniform meshes. The 
extension to higher dimensions and non-uniform structured meshes is straightforward.

\subsection{One-Dimensional Case}\label{sec21}
We begin with the derivation of the 1-D second-order finite-volume method for equation \eqref{1.1}. For simplicity, we divide the 
computational domain into finite-volume cells $C_j=[x_{j-\frac{1}{2}},x_{j+\frac{1}{2}}]$ of a uniform size $\Delta x$ with $x_j=j\Delta x$, $j\in
\{-M,\cdots, M\}$,
and denote by
\begin{equation*}
\overline{\rho}_j(t)=\frac{1}{\Delta x}\int_{C_j}\rho(x,t)\,dx,
\end{equation*}
the computed cell averages of the solution $\rho$, which we assume to be known or approximated at time $t\geq 0$. A semi-discrete finite-volume
scheme is obtained by integrating equation \eqref{1.1} over each cell $C_j$ and is given by the following system of ODEs for 
$\overline{\rho}_j$:
\begin{equation}
\frac{d\overline{\rho}_j(t)}{dt}=-\frac{F_{j+\frac{1}{2}}(t)-F_{j-\frac{1}{2}}(t)}{\Delta x},
\label{a2}
\end{equation}
where the numerical flux $F_{j+\frac{1}{2}}$ approximate the continuous flux $-\rho\xi_x=-\rho(H'(\rho)+V(x)+W\ast\rho)_x$ at cell interface $x_{j+\frac{1}{2}}$ and is constructed next. For simplicity, we will omit the dependence of the computed quantities on $t\geq 0$ in the rest. As in the case of degenerate diffusion equations treated in \cite{Filbet}, we use the upwind numerical fluxes. To this end, we first construct 
piecewise linear polynomials in each cell $C_j$,
\begin{equation}
\widetilde\rho_j(x)=\overline{\rho}_j+(\rho_x)_j(x-x_j),\quad x\in C_j,
\label{rec2}
\end{equation}
and compute the right (``east''), $\rho_j^{\rm E}$, and left (``west''), $\rho_j^{\rm W}$, point values at the cell interfaces 
$x_{j-\frac{1}{2}}$ and $x_{j+\frac{1}{2}}$, respectively:
\begin{equation}
\begin{aligned}
&\rho_j^{\rm E}=\widetilde\rho_j(x_{j+\frac{1}{2}}-0)=\overline\rho_j+\frac{\Delta x}{2}(\rho_x)_j,\\
&\rho_j^{\rm W}=\widetilde\rho_j(x_{j-\frac{1}{2}}+0)=\overline\rho_j-\frac{\Delta x}{2}(\rho_x)_j.
\end{aligned}
\label{rew}
\end{equation}
These values will be second-order accurate provided the numerical derivatives $(\rho_x)_j$ are at least first-order accurate 
approximations of $\rho_x(x,\cdot)$. To ensure that the point values \eqref{rew} are both second-order and nonnegative, the slopes $(\rho_x)_j$ in 
\eqref{rec2} are calculated according to the following adaptive procedure. First, the centered-difference approximations $
(\rho_x)_j=(\overline\rho_{j+1}-\overline\rho_{j-1})/(2\Delta x)$ is used for all $j$. Then, if the reconstructed point values in some cell $C_j$ become negative (i.e., either $\rho_j^{\rm E}<0$ or
$\rho_j^{\rm W}<0$), we recalculate the corresponding slope $(\rho_x)_j$ using a slope limiter, which guarantees that the
reconstructed point values are nonnegative as long as the cell averages $\overline\rho_j$ are nonnegative. In our numerical experiments, we
have used a generalized minmod limiter \cite{LN,NT,Swe,vLeV}:
$$
(\rho_x)_j={\rm minmod}\Big(\theta\,\frac{\overline\rho_{j+1}-\overline\rho_j}{\Delta x},\,
\frac{\overline\rho_{j+1}-\overline\rho_{j-1}}{2\Delta x},\,
\theta\,\frac{\overline\rho_j-\overline\rho_{j-1}}{\Delta x}\Big),
$$
where
$$
{\rm minmod}(z_1,z_2,\ldots):=\left\{\begin{array}{ll}\min(z_1,z_2,\ldots),&\mbox{if}~z_i>0\quad\forall\ i,\\
\max(z_1,z_2,\ldots),&\mbox{if}~z_i<0\quad\forall\ i,\\0,&~\mbox{otherwise},
\end{array}\right.
$$
and the parameter $\theta$ can be used to control the amount of numerical viscosity present in the resulting scheme. In all the numerical 
examples below, $\theta=2$ is used.

Equipped with the piecewise linear reconstruction $\widetilde\rho_j(x)$ and point values $\rho_j^{\rm E},\ \rho_j^{\rm W}$, the upwind 
fluxes in \eqref{a2} are computed as
\begin{equation}
F_{j+\frac{1}{2}}=u_{j+\frac{1}{2}}^+\rho_j^{\rm E}+u_{j+\frac{1}{2}}^-\rho_{j+1}^{\rm W},
\label{nf}
\end{equation}
where the discrete values $u_{j+\frac{1}{2}}$ of the velocities are obtained using the centered-difference approach,
\begin{equation}
u_{j+\frac{1}{2}}=-\frac{\xi_{j+1}-\xi_j}{\Delta x},
\label{vv}
\end{equation}
and the positive and negative parts of $u_{j+\frac{1}{2}}$ are denoted by
\begin{equation}
u_{j+\frac{1}{2}}^+=\max(u_{j+\frac{1}{2}},0),\qquad u_{j+\frac{1}{2}}^-=\min(u_{j+\frac{1}{2}},0).
\label{apm}
\end{equation}
The discrete velocity field $\xi_j$ is calculated by discretizing \eqref{xi}:
\begin{equation}
\xi_j=\Delta x\sum\limits_i W_{j-i}\overline\rho_i+H'(\overline\rho_j)+V_j,
\label{xibar}
\end{equation}
where $W_{j-i}=W(x_j-x_i)$ and $V_j=V(x_j)$. The formula \eqref{xibar} is a second-order approximation of
\begin{equation*}
\sum\limits_i\int_{C_i}W(x_j-s)\widetilde\rho_i(s)\,ds+
H'(\widetilde\rho_j(x_j))+V(x_j).
\end{equation*}
Indeed, the reconstruction \eqref{rec2} yields $H'(\widetilde\rho_j(x_j))=H'(\overline\rho_j)$ and 
\begin{align}
\sum\limits_i\int_{C_i}W(x_j-s)\widetilde\rho_i(s)\,ds
&=\sum\limits_i\overline\rho_i\int_{C_i}W(x_j-s)\,ds
+\sum\limits_i(\rho_x)_i\int_{C_i}W(x_j-s)(s-x_i)\,ds \cr
&=\Delta x\sum\limits_i W_{j-i}\overline\rho_i+\mathcal{O}(\Delta x^2),
\label{tech1}
\end{align}
Here $W_{j-i}$ can be any approximation of the local integral $\frac{1}{\Delta x}
\int_{C_i} W(x_j-s)ds$ with error $O(\Delta x^2)$. If $W$ has a bounded second order derivative near $x_{j-i}$,  $W_{j-i}$ can be chosen to be 
$W(x_{j-i})$ (the middle point rule) or $\big(W(x_{j-i-1/2})+W(x_{j-i+1/2})\big)/2$ (the trapezoidal rule). The integral $\int_{C_i}W(x_j-s)(s-x_i)\,ds$ in the second summation is of $ O(\Delta x^3)$ because of the anti-symmetric factor $s-x_i$, leading to overall error $O(\Delta x^2)$. 

The case with non-smooth or singular interaction potential $W$ has to be treated more carefully.
First,  the last integral $\int_{C_i}W(x_j-s)(s-x_i)\,ds$ in the above formula vanishes as soon as $i=j$ due to the symmetry of $W$ independently of any
possible singularity at $x=x_j$. If $W$ has a locally integrable singularity (usually at the origin), $\frac{1}{\Delta x}\int_{C_i} W(x_j-s)ds$ can still 
be approximated by a higher order quadrature scheme with an error $O(\Delta x^2)$ or smaller. Actually, in the particular case of powers or logarithm 
kernels, it can be explicitly computed. However, the second sum above may 
have a slightly larger error. For instance, if $W(x)\sim |x|^{-\alpha}$ for $0<\alpha<1$, then $\int_{C_i}W(x_j-s)(s-x_i)\,ds \sim O(\Delta x^{2-
\alpha})$ by direct computation when $|i-j|$ is close to zero.   

Finally, the semi-discrete scheme \eqref{a2} is a system of ODEs, which has to be integrated numerically using a stable and accurate ODE  
solver. In all numerical examples reported in next section, the third-order strong preserving Runge-Kutta (SSP-RK) ODE solver 
\cite{GST} is used.

\begin{rmk}\label{rmk1}
The computational bottleneck is the discrete convolution in~\eqref{xibar}. This is a classical problem in scientific computing that can be effectively evaluated using fast convolution algorithms, mainly based Fast Fourier Transforms~\cite{MR2001757}. 
\end{rmk}

\begin{rmk}\label{rmk2}
The second-order finite-volume scheme \eqref{a2}, \eqref{nf}--\eqref{xibar}, reduces to the first-order one if the piecewise constant
reconstruction is used instead of \eqref{rec2}, in which case one has
\begin{equation*}
\widetilde\rho_j(x)=\overline\rho_j,\quad x_j\in C_j,\quad\mbox{and therefore}\quad\rho_j^{\rm E}=\rho_j^{\rm W}=\overline\rho_j,\quad 
\forall j.
\end{equation*}
\end{rmk}

\paragraph{Positivity Preserving.} The resulting scheme preserves positivity of the computed cell averages $\overline{\rho}_j$ as  stated in the 
following theorem. The proof is based on the forward Euler integration of the ODE system \eqref{a2}, but will remain equally valid if the forward 
Euler method were replaced by a higher-order SSP ODE solver~\cite{GST}, whose time step can be expressed as a convex combination of several 
forward Euler steps.

\begin{thm}\label{pos1}
Consider the system \eqref{1.1} with initial data $\rho_0(x)\ge0$ and the semi-discrete finite-volume scheme \eqref{a2}, 
\eqref{nf}--\eqref{xibar} with a positivity preserving piecewise linear reconstruction \eqref{rec2} for $\rho$. Assume that the system of 
ODEs \eqref{a2} is discretized by the forward Euler method. Then, the computed cell averages $\overline{\rho}_j\ge0,\ \forall\ j$, provided that
the following CFL condition is satisfied:
\begin{equation}
\Delta t\le\frac{\Delta x}{2a},\quad\mbox{where}\quad a=\max\limits_j\left\{u_{j+\frac{1}{2}}^+,-u_{j+\frac{1}{2}}^-\right\},
\label{cfl2}
\end{equation}
with $u_{j+\frac{1}{2}}^+$ and $u_{j+\frac{1}{2}}^-$ defined in \eqref{apm}.
\end{thm}
\begin{proof}
Assume that at a given time $t$ the computed solution is known and positive: $\overline{\rho}_j\ge0,\ \forall j$. Then the
new cell averages are obtained from the forward Euler discretization of equation \eqref{a2}:
\begin{equation}
\overline{\rho}_j(t+\Delta t)=\overline{\rho}_j(t)-\lambda\left[F_{j+\frac{1}{2}}(t)-F_{j-\frac{1}{2}}(t)\right],
\label{euler}
\end{equation}
where $\lambda:={\Delta t}/{\Delta x}$. As above, the dependence of all terms on the RHS of \eqref{euler} on $t$ is suppressed in the following 
to simplify the notation. Using \eqref{nf} and the fact that $\overline{\rho}_j=\frac{1}{2}\left(\rho_j^{\rm E}+\rho_j^{\rm W}\right)$ (see
\eqref{rew}), we obtain
\begin{equation}
\begin{aligned}
\overline{\rho}_j(t+\Delta t)&=\frac{1}{2}\left(\rho_j^{\rm E}+\rho_j^{\rm W}\right)
-\lambda\left[u_{j+\frac{1}{2}}^+\rho_{j}^{\rm E}+u_{j+\frac{1}{2}}^-\rho_{j+1}^{\rm W}-
u_{j-\frac{1}{2}}^+\rho_{j-1}^{\rm E}-u_{j-\frac{1}{2}}^-\rho_j^{\rm W}\right]\\
&=\lambda u_{j-\frac{1}{2}}^+\rho_{j-1}^{\rm E}+\left(\frac{1}{2}-\lambda u_{j+\frac{1}{2}}^+\right)\rho_{j}^{\rm E}
+\left(\frac{1}{2}+\lambda u_{j-\frac{1}{2}}^-\right)\rho_{j}^{\rm W}-\lambda u_{j+\frac{1}{2}}^-\rho_{j+1}^{\rm W}.
\end{aligned}
\label{euler1}
\end{equation}
It follows from \eqref{euler1} that the new cell averages $\overline{\rho}_j(t+\Delta t)$ are linear combinations of the nonnegative 
reconstructed point values $\rho_{j-1}^{\rm E},\rho_j^{\rm E},\rho_{j}^{\rm W}$ and $\rho_{j+1}^{\rm W}$. Since $u_{j-\frac{1}{2}}^+\ge0$
and $u_{j+\frac{1}{2}}^-\le0$, we conclude that $\overline{\rho}_j(t+\Delta t)\ge0,\ \forall j$, provided that the CFL condition \eqref{cfl2} is 
satisfied.
\end{proof}

\begin{rmk}
Similar result holds for the first-order finite-volume scheme with the CFL condition reduced to
$$
\Delta t\le\frac{\Delta x}{2\max\limits_j\left(u_{j+\frac{1}{2}}^+-u_{j-\frac{1}{2}}^-\right)}.
$$
\end{rmk}

\paragraph{Discrete Entropy Dissipation.} A discrete version of the entropy $E$ defined in \eqref{fe} is given by
\begin{equation}
E_\Delta(t)=\Delta x\sum\limits_j\left[\frac12\Delta x\sum\limits_i W_{j-i}\overline{\rho}_i\overline{\rho}_j+H(\overline\rho_j)+
{V}_j\overline{\rho}_j\right].
\label{dfe}
\end{equation}
We also introduce the discrete version of the entropy dissipation
\begin{equation}
I_\Delta(t)=\Delta x\sum\limits_j (u_{j+\frac{1}{2}})^2\min\limits_j(\rho_j^{\rm E},\rho_{j+1}^{\rm W}).
\label{ded}
\end{equation}
In the following theorem, we prove that the time derivative of $E_\Delta(t)$ is less or equal than the negative of $I_\Delta(t)$, mimicking the corresponding property of the continuous relation.

\begin{thm}\label{dd1}
Consider the system \eqref{1.1} with no flux boundary conditions on $[-L,L]$ with $L>0$ and with initial data $\rho_0(x)\ge0$. Given the semi-discrete finite-volume scheme \eqref{a2} with $\Delta x=L/M$,
\eqref{nf}--\eqref{xibar} with a positivity preserving piecewise linear reconstruction \eqref{rec2} for $\rho$ and discrete boundary conditions 
$F_{M+\frac{1}{2}}=F_{-M-\frac{1}{2}}=0$. Then,
\begin{equation*}
\frac{d}{dt}E_\Delta(t)\le-I_\Delta(t),\quad \forall t>0.
\end{equation*}
\end{thm}
\begin{proof}
We start by differentiating \eqref{dfe} with respect to time to obtain:
\begin{equation*}
\begin{aligned}
\frac{d}{dt}E_\Delta(t)&=\Delta x\sum\limits_j\left[\Delta x\sum\limits_i W_{j-i}\overline{\rho}_i\frac{d\overline{\rho}_j}{dt}
+H'(\overline\rho_j)\frac{d\overline{\rho}_j}{dt}+{V}_j\frac{d\overline{\rho}_j}{dt}\right]\\
&=\Delta x\sum\limits_j\left[\Delta x\sum\limits_i W_{j-i}\overline{\rho}_i+H'(\overline\rho_j)+
{V}_j\right]\frac{d\overline{\rho}_j}{dt}.
\end{aligned}
\end{equation*}
Using the definition \eqref{xibar} and the numerical scheme \eqref{a2}, we have
$$
\frac{d}{dt}E_\Delta(t)=-\Delta x\sum\limits_j\xi_j\,\frac{F_{j+\frac{1}{2}}-F_{j-\frac{1}{2}}}{\Delta x}.
$$
A discrete integration by parts using the no flux discrete boundary conditions along with \eqref{vv} yields
$$
\frac{d}{dt}E_\Delta(t)=-\sum\limits_j(\xi_j-\xi_{j+1})F_{j+\frac{1}{2}}=
-\Delta x\sum\limits_j u_{j+\frac{1}{2}}F_{j+\frac{1}{2}}.
$$
Finally, using the definition of the upwind fluxes \eqref{nf} and formulas \eqref{apm} and \eqref{ded}, we conclude
$$
\frac{d}{dt}E_\Delta(t)=-\Delta x\sum\limits_j u_{j+\frac{1}{2}}
\left[u_{j+\frac{1}{2}}^+\rho_j^{\rm E}+u_{j+\frac{1}{2}}^-\rho_{j+1}^{\rm W}\right]
\le-\Delta x\sum\limits_j(u_{j+\frac{1}{2}})^2\min\limits_j(\rho_j^{\rm E},\rho_{j+1}^{\rm W})=-I_\Delta(t).
$$
\end{proof}


\subsection{Two-Dimensional Case}\label{sec22}
In this subsection, we quickly describe a semi-discrete second-order finite-volume method for the 2-D equation \eqref{1.1}. We explain the main ideas in 2D for the sake of the reader. As already mentioned, the first order scheme generalizes easily to unstructured meshes. However, higher 
order schemes with the desired entropy decreasing property are harder to obtain in this setting for higher dimensions.
We introduce a  Cartesian mesh consisting of the cells $C_{j,k}:=[x_{j-\frac{1}{2}},x_{j+\frac{1}{2}}]\times[y_{k-\frac{1}{2}},y_{k+\frac{1}{2}}]$, which for the
sake of simplicity are assumed to be of the uniform size $\Delta x\Delta y$, that is, 
$x_{j+\frac{1}{2}}-x_{j-\frac{1}{2}}\equiv\Delta x,\ \forall\ j$, and $y_{k+\frac{1}{2}}-y_{k-\frac{1}{2}}\equiv\Delta y,\ \forall\ k$.

A general semi-discrete finite-volume scheme for equation \eqref{1.1} can be written in the following form:
\begin{equation}
\frac{d\overline\rho_{j,k}}{dt}=-\frac{F^x_{j+\frac{1}{2},k}-F^x_{j-\frac{1}{2},k}}{\Delta x}
-\frac{F^y_{j,k+\frac{1}{2}}-F^y_{j,k-\frac{1}{2}}}{\Delta y}.
\label{fv2}
\end{equation}
Here, we define $$\bar\rho_{j,k}(t)\approx\dfrac{1}{\Delta x\Delta y}\iint_{C_{j,k}}\rho(x,y,t)dxdy$$ as the cell averages of the 
computed solution and $F^x_{j+\frac{1}{2},k}$ and $F^y_{j,k+\frac{1}{2}}$ are upwind numerical fluxes in the $x$ and $y$ directions, 
respectively.

As in the 1-D case, to obtain formulae for numerical fluxes, we first compute $\rho_{j,k}^{\rm E},\rho_{j,k}^{\rm W},\rho_{j,k}^{\rm N}$ 
and $\rho_{j,k}^{\rm S}$, which are one-sided point values of the piecewise linear reconstruction
\begin{equation}
\widetilde\rho(x,y)=\overline{\rho}_{j,k}+(\rho_x)_{j,k}(x-x_j)+(\rho_y)_{j,k}(y-y_k),\quad (x,y)\in C_{j,k},
\label{rec22}
\end{equation}
at the cell interfaces $(x_{j+\frac{1}{2}},y_k)$, $(x_{j-\frac{1}{2}},y_k)$, $(x_{j},y_{k+\frac{1}{2}})$, $(x_{j},y_{k-\frac{1}{2}})$, 
respectively. Namely,
\begin{equation}
\begin{aligned}
&\rho_{j,k}^{\rm E}:=\widetilde\rho(x_{j+\frac{1}{2}}-0,y_k)=\overline\rho_{j,k}+\frac{\Delta x}{2}(\rho_x)_{j,k},\quad
\rho_{j,k}^{\rm W}:=\widetilde\rho(x_{j-\frac{1}{2}}+0,y_k)=\overline\rho_{j,k}-\frac{\Delta x}{2}(\rho_x)_{j,k},\\[0.5ex]
&\rho_{j,k}^{\rm N}:=\widetilde\rho(x_{j},y_{k+\frac{1}{2}}-0)=\overline\rho_{j,k}+\frac{\Delta y}{2}(\rho_y)_{j,k},\quad
\rho_{j,k}^{\rm S}:=\widetilde\rho(x_{j},y_{k-\frac{1}{2}}+0)=\overline\rho_{j,k}-\frac{\Delta y}{2}(\rho_y)_{j,k}.
\end{aligned}
\label{pv2}
\end{equation}
To ensure the point values in \eqref{pv2} are both second-order and nonnegative, the slopes in \eqref{rec22} are calculated according to 
the adaptive procedure similarly to the 1-D case. First, the centered-difference approximations,
\begin{equation*}
(\rho_x)_{j,k}=\frac{\overline\rho_{j+1,k}-\overline\rho_{j-1,k}}{2\Delta x}\quad\mbox{and}\quad
(\rho_y)_{j,k}=\frac{\overline\rho_{j,k+1}-\overline\rho_{j,k-1}}{2\Delta y}
\end{equation*}
are used for all $j,k$. Then, if the reconstructed point values in some cell $C_{j,k}$ become negative, we recalculate the corresponding 
slopes $(\rho_x)_{j,k}$ or $(\rho_y)_{j,k}$ using a monotone nonlinear limiter, which guarantees that the reconstructed point values are 
nonnegative as long as the cell averages of $\overline\rho_{j,k}$ are nonnegative for all $j,k$. In our numerical experiments, we have used
the one-parameter family of the generalized minmod limiters with $\theta\in[1,2]$:
\begin{equation*}
\begin{aligned}
(\rho_x)_{j,k}=\textrm{minmod}\left(\theta\frac{\overline\rho_{j,k}-\overline\rho_{j-1,k}}{\Delta x},
\frac{\overline\rho_{j+1,k}-\overline\rho_{j-1,k}}{2\Delta x},\theta\frac{\overline\rho_{j+1,k}-\overline\rho_{j,k}}{\Delta x}\right), \\
(\rho_y)_{j,k} =\textrm{minmod}\left(\theta\frac{\overline\rho_{j,k}-\overline\rho_{j,k-1}}{\Delta y},
\frac{\overline\rho_{j,k+1}-\overline\rho_{j,k-1}}{2\Delta y},\theta\frac{\overline\rho_{j,k+1}-\overline\rho_{j,k}}{\Delta y}\right).
\end{aligned}
\end{equation*}

Given the polynomial reconstruction \eqref{rec22} and its point values \eqref{pv2}, the upwind numerical fluxes in \eqref{fv2} are defined 
as
\begin{equation}
F_{j+\frac{1}{2},k}^x=u_{j+\frac{1}{2},k}^+\rho_{j,k}^{\rm E}+u_{j+\frac{1}{2},k}^-\rho_{j+1,k}^{\rm W},\qquad
F_{j,k+\frac{1}{2}}^y=v_{j,k+\frac{1}{2}}^+\rho_{j,k}^{\rm N}+v_{j,k+\frac{1}{2}}^-\rho_{j,k+1}^{\rm S},
\label{nfx}
\end{equation}
where
\begin{equation*}
u_{j+\frac{1}{2},k}=-\frac{\xi_{j+1,k}-\xi_{j,k}}{\Delta x}, \qquad 
v_{j,k+\frac{1}{2}}=-\frac{\xi_{j,k+1}-\xi_{j,k}}{\Delta y},
\end{equation*}
the  values of $u_{j+\frac{1}{2},k}^{\pm}$ and $v_{j,k+\frac{1}{2}}^{\pm}$ are defined according to \eqref{apm}, and
\begin{equation}
\xi_{j,k}=\Delta x\Delta y\sum\limits_{i,l} W_{j-i,k-l}\overline\rho_{i,l}+H'(\overline\rho_{j,k})+V_{j,k}.
\label{xibar2}
\end{equation}
Here, $W_{j-i,k-l}=W(x_j-x_i,y_k-y_l)$ and $V_{j,k}=V(x_j,y_k)$.

Similarly to the 1-D case, the formula \eqref{xibar2} for $\xi_{j,k}$ is obtained by using the reconstruction formula \eqref{rec22} and 
applying the midpoint quadrature rule to the first integral in
$$
\xi_{j,k}=
\sum\limits_{i,l}\iint_{C_{i,l}}W(x-s,y-r)\widetilde\rho_{i,l}(s,r)\,ds\,dr\\
+H'(\widetilde\rho_{j,k}(x,y))+V(x_j,y_k).
$$
As in the 1-D case, the ODE system \eqref{fv2} is to be integrated numerically by a stable and sufficiently accurate ODE solver such as the third-order SSP-RK ODE solver \cite{GST}.

\begin{rmk}\label{rmk23}
As in the 1-D case, the first-order finite-volume method is obtained by taking
$$
\widetilde\rho_{j,k}(x,y)=\overline\rho_{j,k}\quad{\rm and}\quad
\rho_{j,k}^{\rm E}=\rho_{j,k}^{\rm W}=\rho_{j,k}^{\rm N}=\rho_{j,k}^{\rm S}=\overline\rho_{j,k},\quad\forall j,k.
$$
\end{rmk}

\paragraph{Positivity Preserving.} The resulting 2-D finite-volume scheme will preserve positivity of the computed cell averages $\overline{\rho}
_{j,k},\ \forall j,k$, as long as an SSP ODE solver, whose time steps are convex combinations of forward Euler steps, is used for time 
integration. We omit the proof of the positivity property of the scheme as it follows exactly the lines of Theorem \ref{pos1}. The only difference is 
that in the 2-D case 
$\overline\rho_{j,k}=\tfrac{1}{4}\left(\rho_{j,k}^{\rm E}+\rho_{j,k}^{\rm W}+\rho_{j,k}^{\rm N}+\rho_{j,k}^{\rm S}\right)$, which leads
to a slightly modified CFL condition. We thus have the following theorem.

\begin{thm}\label{pos2}
Consider the system \eqref{1.1} with initial data $\rho_0(x)\ge0$ and the semi-discrete finite-volume scheme \eqref{fv2},
\eqref{nfx}--\eqref{xibar2} with a positivity preserving piecewise linear reconstruction \eqref{rec22} for $\rho$. Assume that the system of
ODEs \eqref{fv2} is discretized by the forward Euler (or a strong stability preserving Runge-Kutta) method. Then, the computed cell 
averages $\overline{\rho}_{j,k}\ge0,\ \forall j,k$, provided the following CFL condition is satisfied:
\begin{equation*}
\Delta t\le\min\left\{\frac{\Delta x}{4a},\frac{\Delta y}{4b}\right\},
\quad a=\max\limits_{j,k}\left\{u_{j+\frac{1}{2},k}^+,-u_{j+\frac{1}{2},k}^-\right\},\quad
b=\max\limits_{j,k}\left\{v_{j,k+\frac{1}{2}}^+,-v_{j,k+\frac{1}{2}}^-\right\},
\end{equation*}
where $u_{j+\frac{1}{2},k}^{\pm}$ and $v_{j,k+\frac{1}{2}}^{\pm}$ are defined according to \eqref{apm}.
\end{thm}

\paragraph{Discrete Entropy Dissipation.} We define the discrete entropy
\begin{equation*}
E_\Delta(t)=\Delta x\Delta y\sum\limits_{j,k}\left[\frac{1}{2}\Delta x\Delta y\sum\limits_{i,l} W_{j-i,k-l}\overline{\rho}_{i,l}\overline{\rho}_{j,k}
+H(\overline\rho_{j,k})+\overline{V}_{j,k}\overline{\rho}_{j,k}\right],
\end{equation*}
and discrete entropy dissipation
\begin{equation*}
I_\Delta(t)=\Delta x\Delta y\sum\limits_{j,k}\left[(u_{j+\frac{1}{2},k})^2+(v_{j,k+\frac{1}{2}})^2\right]\min\limits_{j,k}
\left(\rho_{j,k}^{\rm E},\rho_{j+1,k}^{\rm W},\rho_{j,k}^{\rm N},\rho_{j,k+1}^{\rm S}\right).
\end{equation*}
Similarly to the 1-D case, we can show the following dissipative property of the scheme.

\begin{thm}\label{dd2}
Consider the system \eqref{1.1} with no flux boundary conditions in the domain $[-L,L]^2$ with $L>0$ and with initial data $\rho_0(x)\ge0$. 
Given the semi-discrete finite-volume scheme \eqref{fv2},
\eqref{nfx}--\eqref{xibar2} with a positivity preserving piecewise linear reconstruction \eqref{rec22} for $\rho$, with $\Delta x=L/M$, and with 
discrete no-flux boundary conditions $F^x_{M+\frac{1}{2},k} = {F^x_{-M-\frac{1}{2},k}} = F^y_{j,M+\frac{1}{2}} = { F^y_{j,-M-\frac{1}{2}}}=0$. Then,
\begin{equation*}
\frac{d}{dt}E_\Delta(t)\le-I_\Delta(t),\quad \forall t>0.
\end{equation*}
\end{thm}


\section{Numerical Experiments}\label{sec3}

In this section, we present several numerical examples, focusing mainly on the steady states or long time behaviors of the solutions to the general equation
\[
 \rho_t=\nabla\cdot\big[\rho\nabla\big(H'(\rho)+V(\mathbf{x})+
    W\ast\rho\big)\big],
\quad \mathbf{x}\in\mathbb{R}^d,\ t>0.
\]
A previous detailed study in \cite{Filbet} for the degenerate parabolic and drift-diffusion equations demonstrated the good performance of the method (with small variants) in dealing with exponential rates of convergence toward compactly supported 
Barenblatt solutions. Here we will concentrate mostly on cases with the  interaction potential $W$, and show that key properties like 
non-negativity and entropy dissipation are preserved. 
We will first start our discussion by using some test cases to validate the order of convergence of the scheme in space and its relation to the regularity of the steady states. 
If the solution $\rho$ is smooth, the spatial discretization given in Section~\ref{sec2} is shown to be of second order. However, in practice, the steady states of~\eqref{1.1} are usually compactly supported, with discontinuities in the derivatives or even in the solutions themselves near the boundaries. This loss of regularity  of the steady states usually leads to degeneracy in the order of convergence, as shown in Examples \ref{ex31}--\ref{ex34}.
Then, we will illustrate with several examples that the presented finite-volume scheme can be 
used for a numerical study of many challenging questions in which theoretical analysis has not yet been fully developed.

\subsection{Steady states: Spatial Order and Time Stabilization}\label{sec31}
\begin{example}[{\bf Attractive-repulsive kernels}]\label{ex31}
We first consider equation \eqref{1.1} in 1-D with only the interaction kernel $W(x) = |x|^2/2-\log |x|$ (i.e., $H(\rho)=0, V(x)=0$). The 
corresponding unit-mass steady state is given by (see \cite{MR1485778}):
$$
\rho_\infty(x) = \begin{cases}
\frac{1}{\pi}\sqrt{2-x^2},\qquad &|x|\leq \sqrt{2},\cr
0, &\mbox{otherwise}\,.
\end{cases}
$$
and is  H\"{o}lder continuous with exponent $\alpha=\frac{1}{2}$.
This steady state is the unique global minimizer of the free energy~\eqref{fe} and it approached by the solutions of~\eqref{1.1} with an exponential convergence rate as shown in~\cite{CFP}. We compute $\rho_\infty$ by numerically solving \eqref{1.1} at large time, with the initial condition $\rho(x,0)=\frac{1}{\sqrt{2\pi}}e^{-x^2/2}$. In Figure \ref{fig:quadlog_prof}{\bf (a)}, we plot the numerical steady state 
obtained on a very coarse grid with $\Delta x=\sqrt{2}/5$. As one can see, even on such a coarse grid, the numerical steady state is in  good agreement with the exact one, except near the boundary $x = \pm \sqrt{2}$. The spatial convergence error of the steady states in $L^1$ norm and $L^\infty$  norm is shown in Figure~\ref{fig:quadlog_prof}{\bf (b)}.  As a general rule, the practical convergence error of the numerical steady state is $\alpha$ in $L^\infty$ norm and $\alpha+1$ in $L^1$ norm, if 
the exact steady state is $C^\alpha$-H\"older continuous. 

\begin{figure}[ht!]
  \begin{center}
\includegraphics[totalheight=0.25\textheight]{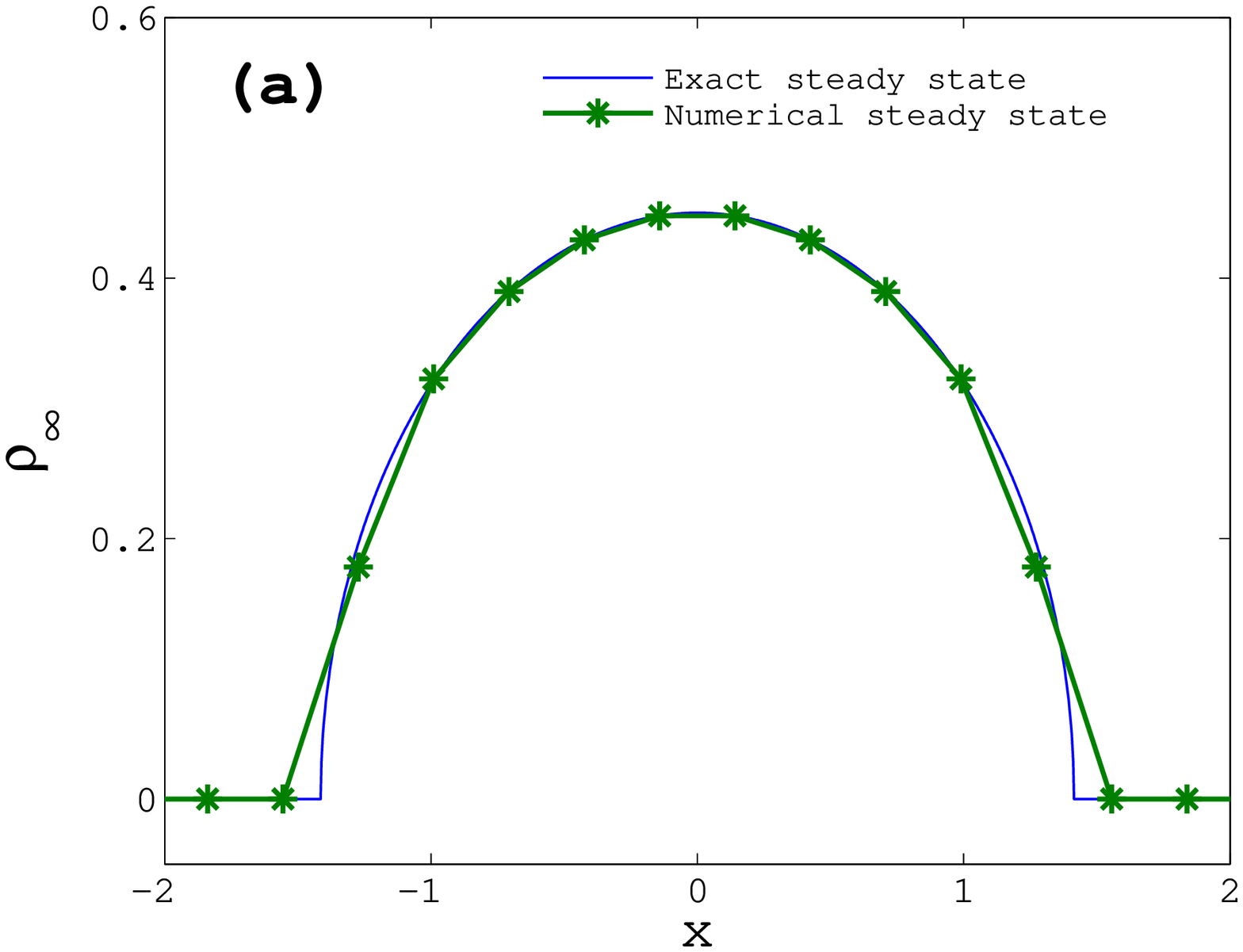}
$~~~$
\includegraphics[totalheight=0.25\textheight]{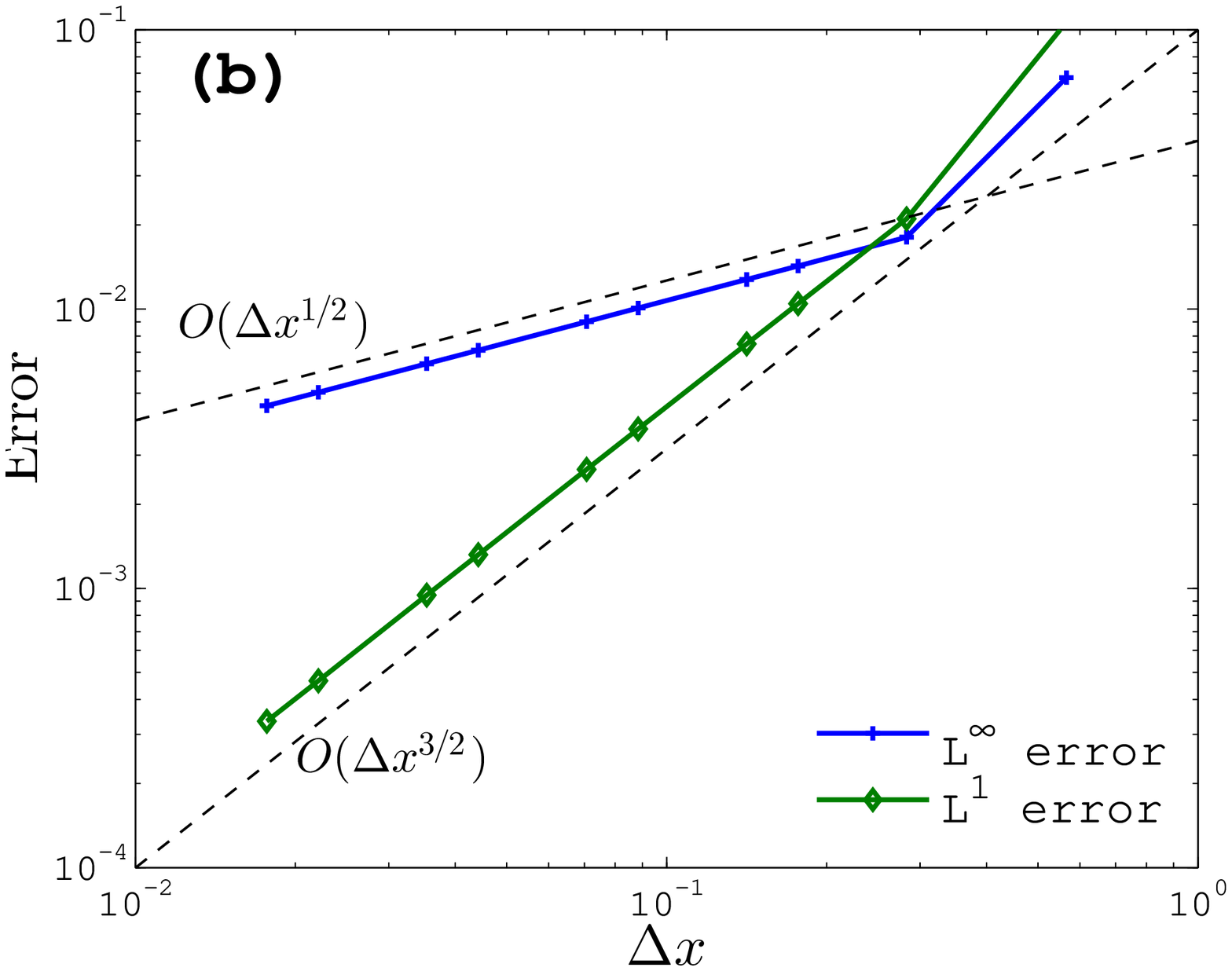}
\end{center}    
\caption{{\bf (a)} The numerical steady state with grid size $\Delta x = \sqrt{2}/5$, compared with the exact one. {\bf (b)} The convergence of error in $L^1$ and $L^\infty$ norms. Here the $L^1$ norm is computed by taking the numerical steady state piecewise constant inside each cell and $L^\infty$ norm is evaluated only at the 
cell centers.}
\label{fig:quadlog_prof}
\end{figure}
\end{example}

\begin{example}[{\bf Nonlinear diffusion with nonlocal attraction kernel}]\label{ex32} 
Next, we consider the equation \eqref{1.1} in 1-D with 
$H(\rho) = \frac{\nu}{m}\rho^m$, $W(x) = W(|x|)$ and $V(x) = 0$,
where $\nu>0, m>1$ and $W\in \mathcal{W}^{1,1}(\mathbb{R})$ is an increasing function on
$[0,\infty)$, i.e.,
\begin{equation}\label{eq:nldiff}
    \rho_t= \big( \rho (\nu\rho^{m-1}+W\ast \rho)_x\big)_x.
\end{equation}
This equation arises in  some physical and biological modelling with competing nonlinear diffusion and nonlocal attraction, see \cite{MR2257718} for instance. The attraction represented by convolution $W*\rho$ is relatively weak (compared to that in the Keller-Segel model discussed below), and the solution does not blow up with bounded initial data, while the long time behavior of the solution is characterized by an extensive study of the steady states in~\cite{BFH}. When $m>2$, the attraction dominates the nonlinear diffusion, leading to a compactly supported steady state. When $m<2$, the behavior depends on the diffusion coefficient $\nu$: there is a local steady state for small $\nu$ with localized initial data and the solution always decays to zero for large $\nu$. The borderline case $m=2$ is investigated in~\cite{BDF} for non-compactly supported kernels, where the evolution depends on the coefficient $\nu$, the total conserved mass, and $\|W\|_1$.

We begin by numerically calculating the solutions to the 1-D equation \eqref{eq:nldiff} with nonlinear diffusion and $W(x)=-e^{-|x|^2/2\sigma}/\sqrt{2\pi\sigma}$, for some constants $\sigma>0$. The corresponding steady states can also be obtained by implementing an iterative procedure proposed in~\cite{BFH}. Here, we compute the steady state solutions $\rho_\infty$ by the time evolution of \eqref{eq:nldiff} subject to Gaussian-type initial data 
$$
\rho_0(x)=\frac{1}{\sqrt{8\pi}}\left[e^{-0.5(x-3)^2}+e^{-0.5(x+3)^2}\right].
$$ 
The simulations are run on the computational domain $[-6,6]$ with the mesh size $\Delta x =0.02$ for large time until stabilization and the results are plotted in Figure \ref{fig:aggdiff1d_prof}{\bf (a)} for different values of $m$.  As one can observe, the boundary behavior of the 
compactly supported steady states has a similar dependence on $m$ as the Barenblatt solutions of the classical porous medium equation 
$\rho_t = \nu\big(\rho(\rho^{m-1})_x\big)_x$, that is, only H\"{o}lder continuous with exponent $\alpha=\min\big(1,1/(m-1)\big)$. Using the 
steady states of~\eqref{1.1} computed by the iterative scheme proposed in~\cite{BFH}, we can  check the spatial convergence 
error of our scheme on different grid sizes $\Delta x$. As shown in Figure~\ref{fig:aggdiff1d_prof}{\bf (b)}, the spatial convergence error of the steady states is $\min\big(2,m/(m-1)\big)$ in $L^1$ norm and is 
$\min\big(1,1/(m-1)\big)$ in $L^\infty$ norm.
\begin{figure}[ht!]
\begin{center}
\includegraphics[totalheight=0.25\textheight]{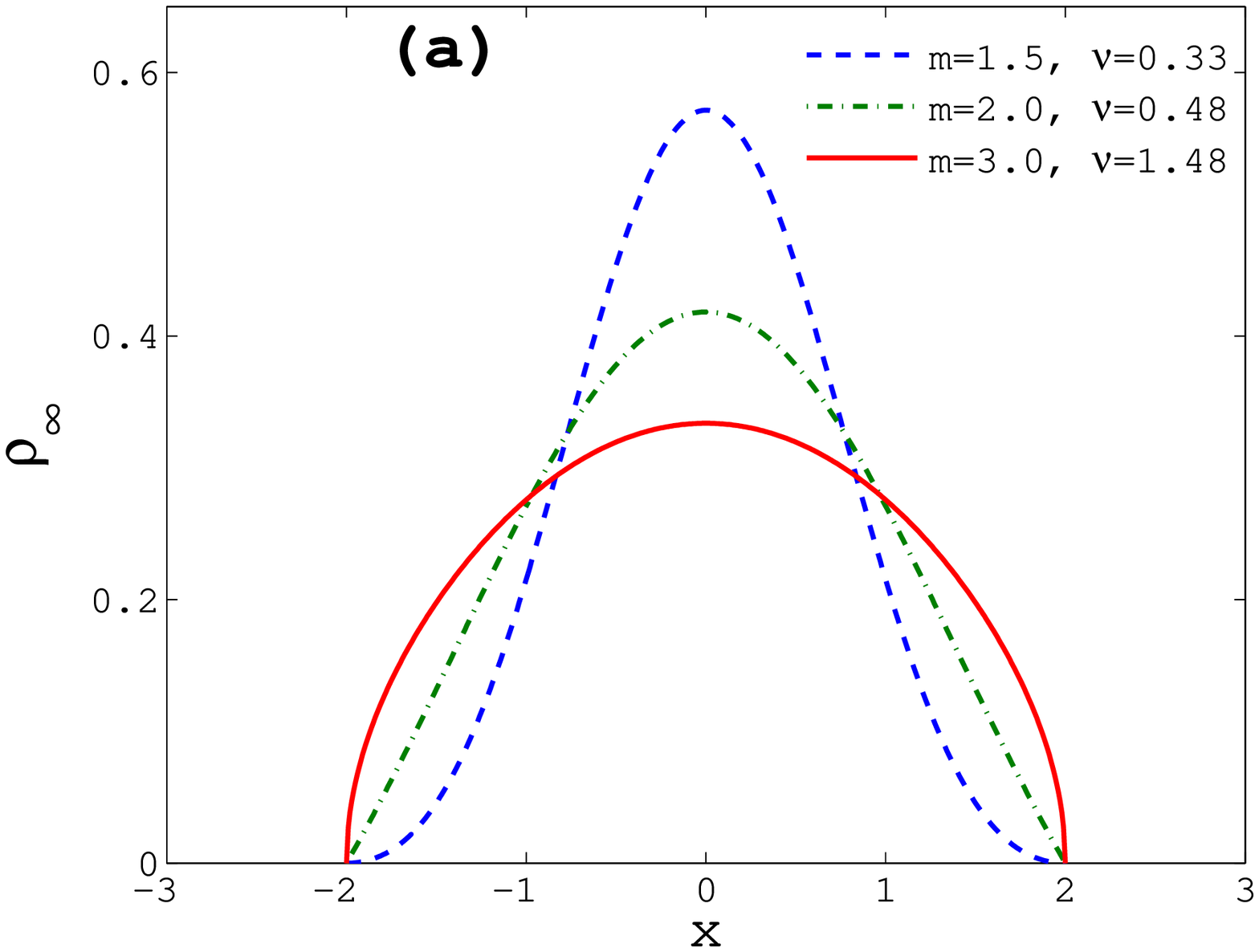} 
$~~~$
\includegraphics[totalheight=0.25\textheight]{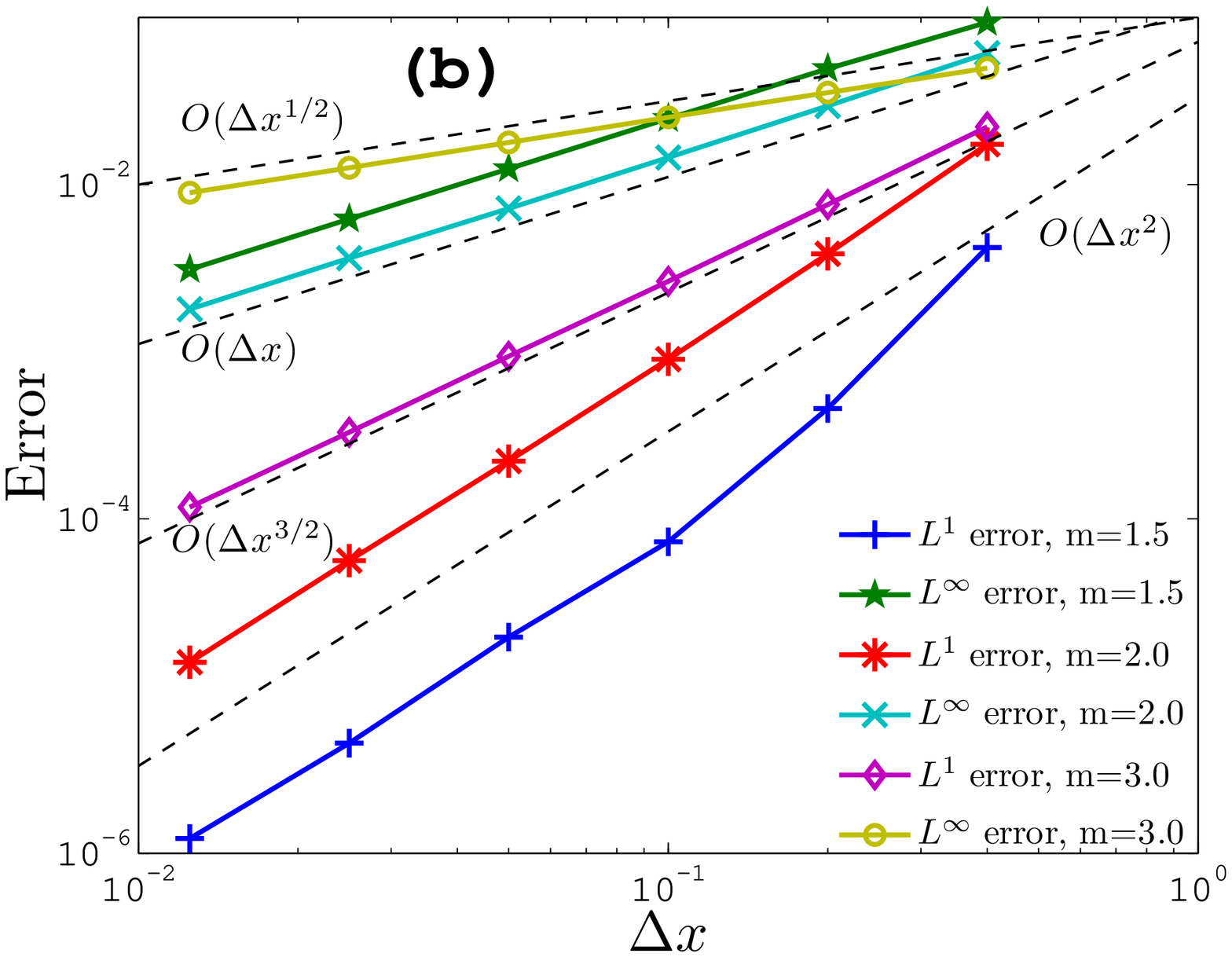}
\end{center}    
\caption{{\bf (a)} The steady states with unit total mass for
different $m$ have H\"{o}lder exponent $\alpha=\min\big(1,1/(m-1)\big)$ and $\sigma=1$,
where $\nu$ is chosen such that the corresponding $\rho_\infty$ is supported on $[-2,2]$. {\bf (b)} The convergence of the steady states $\rho_\infty$ on different grid size 
$\Delta x$, which is $\min\big(2,m/(m-1)\big)$ in $L^1$ norm and is $\min\big(1,1/(m-1)\big)$ in $L^\infty$ norm.}
\label{fig:aggdiff1d_prof}
\end{figure}

Now let us turn our attention to the time evolution and the stabilization in time toward equilibria and show that the convergence in time toward equilibration can be arbitrarily slow. This is due to the fact that the effect of attraction is very small for large distances. Actually, 
different bumps at large distances will slowly diffuse and take very long time to attract each other. However, once they reach certain distance, the convexity of the Gaussian well will lead to equilibration exponentially fast in time. These two different time scales can be observed in Figure 
\ref{fig:aggdiff1d_prof2}, where the time energy decay and density evolution are plotted to the solution corresponding to $m=3, \sigma=1$, and 
$\nu=1.48$ (see also Figure \ref{fig:aggdiff1d_prof}).
\begin{figure}[ht]
\begin{center}
\includegraphics[totalheight=0.25\textheight]{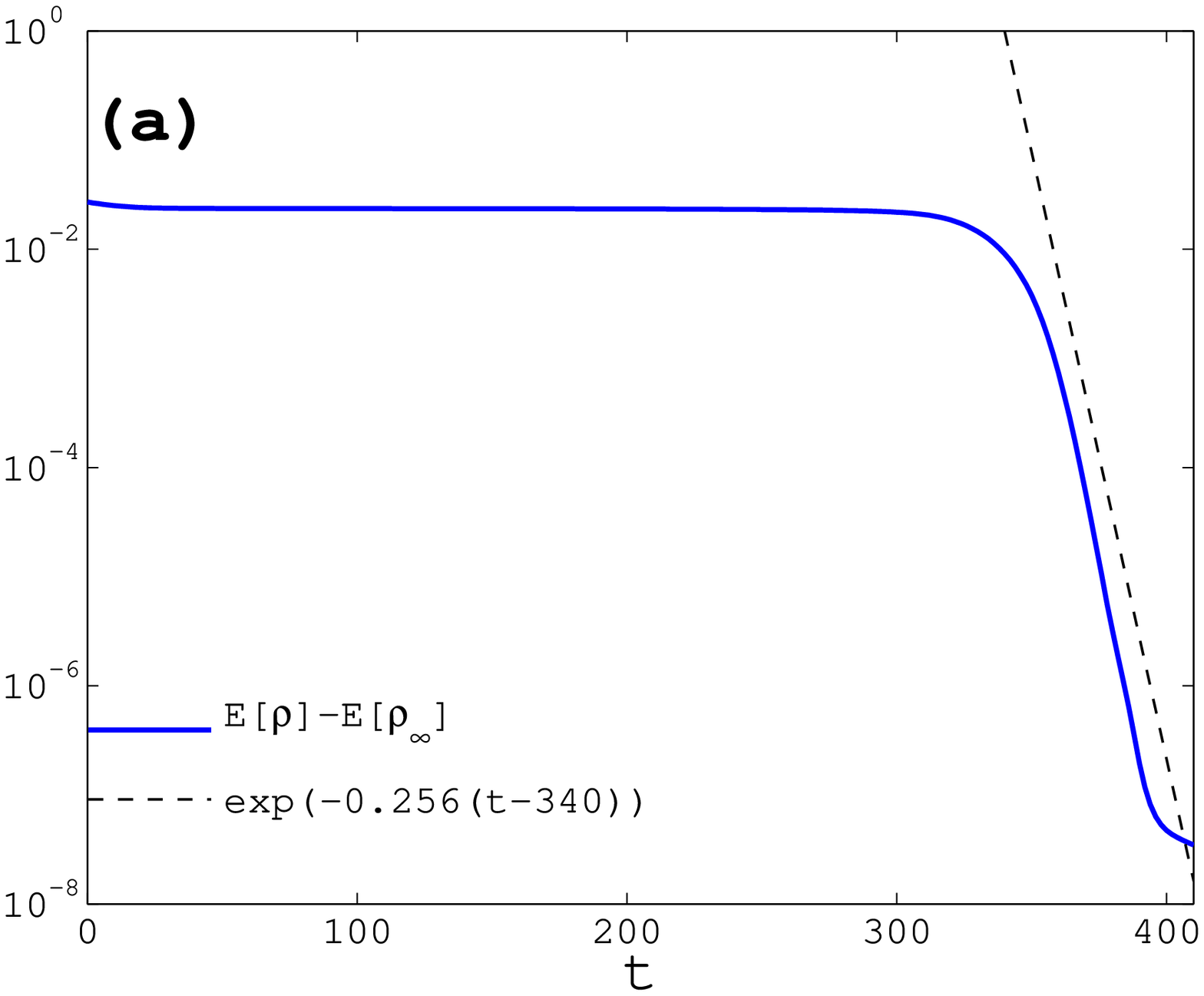} 
$~~~$
\includegraphics[totalheight=0.25\textheight]{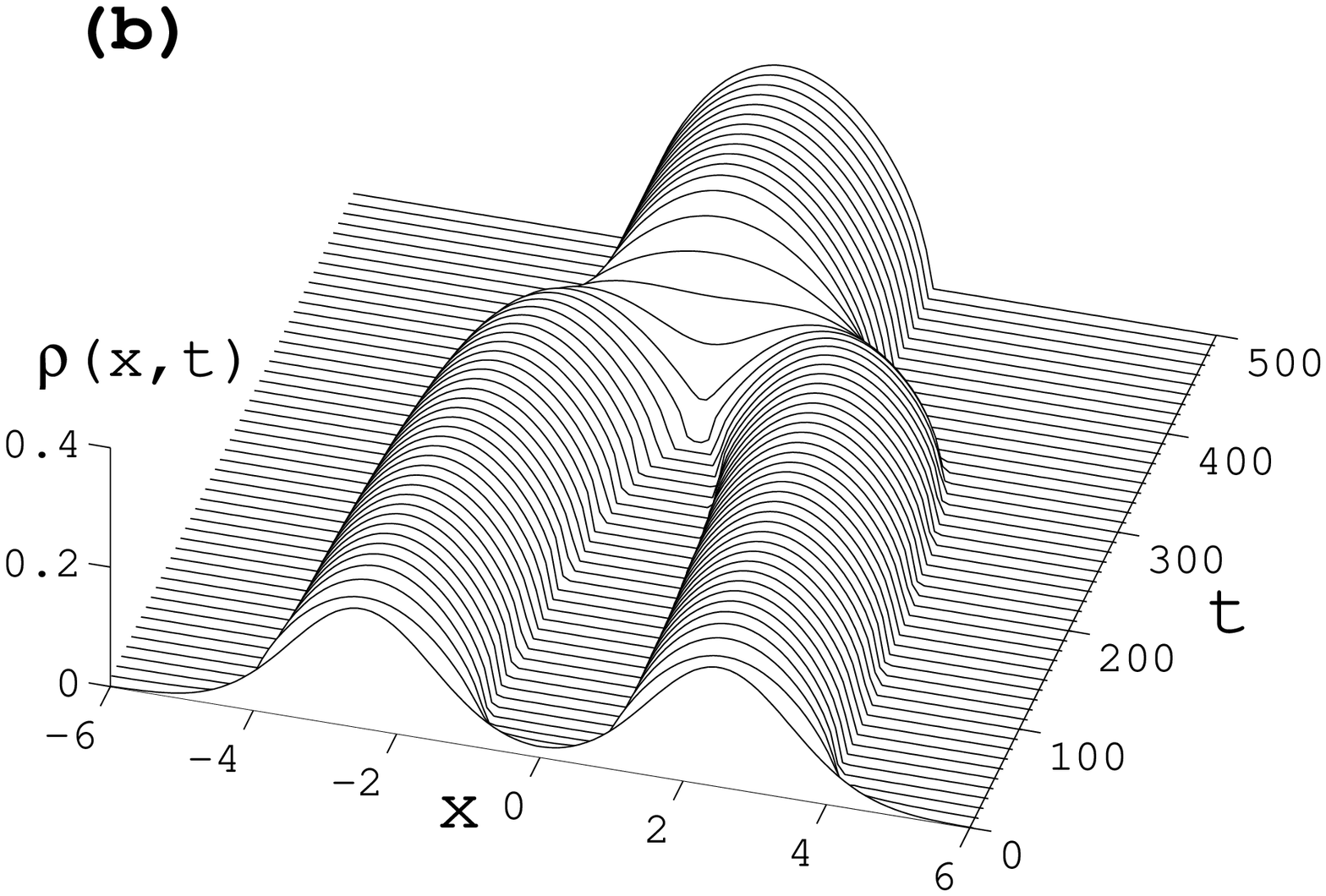}
\end{center}    
\caption{{\bf (a)} The two timescales in the decay towards the unique equilibrium solution corresponding: very slow energy decay 
followed by an exponential decay. {\bf (b)} Time evolution of the density. Here, $m=3, \sigma=1$ and $\nu=1.48$.
}
\label{fig:aggdiff1d_prof2}
\end{figure}
\end{example}

\begin{example}[{\bf Nonlinear diffusion with compactly supported attraction kernel}]\label{ex33} 
The dynamics of the solution of the 1-D equation \eqref{1.1} with characteristic 
functions as initial data is shown in Figure~\ref{fig:nldiff}, for the compactly supported interaction kernel $W(x)=-(1-|x|)_+$. For $\rho_0(x)=
\chi_{[-2,2]}(x)$, the solution forms two bumps and then 
merges to a single one, which is the global minimizer of the energy. When $\rho_0(x)=\chi_{[-3,3]}(x)$, the solution converges to three non-interacting bumps (in the sense that $\tfrac{\partial\xi}{\partial x} \rho \equiv 0$),
each of which is a steady state. 

\begin{figure}[htp]
\begin{center}
\includegraphics[totalheight=0.22\textheight]{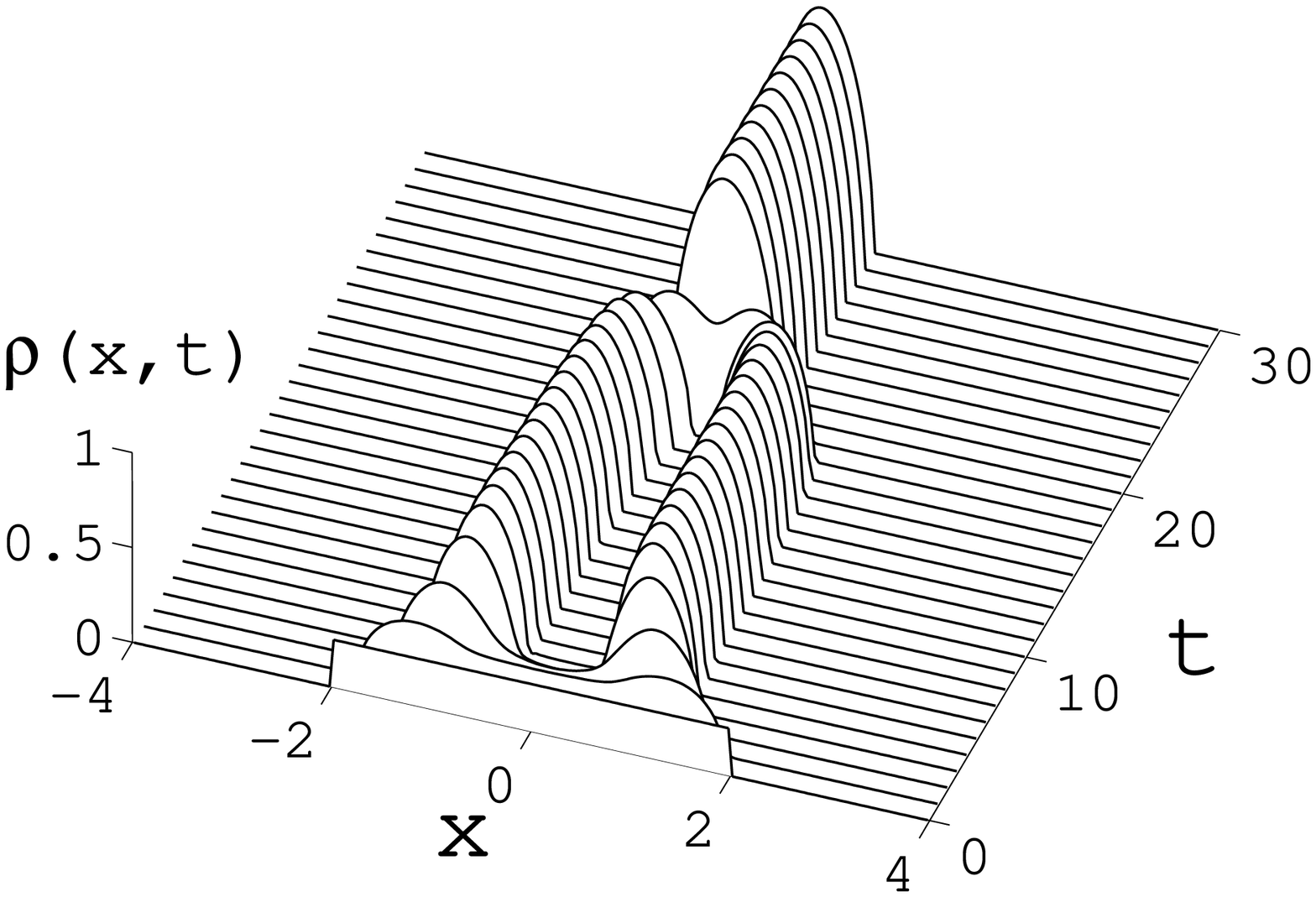}
$~ ~ ~ $        
\includegraphics[totalheight=0.22\textheight]{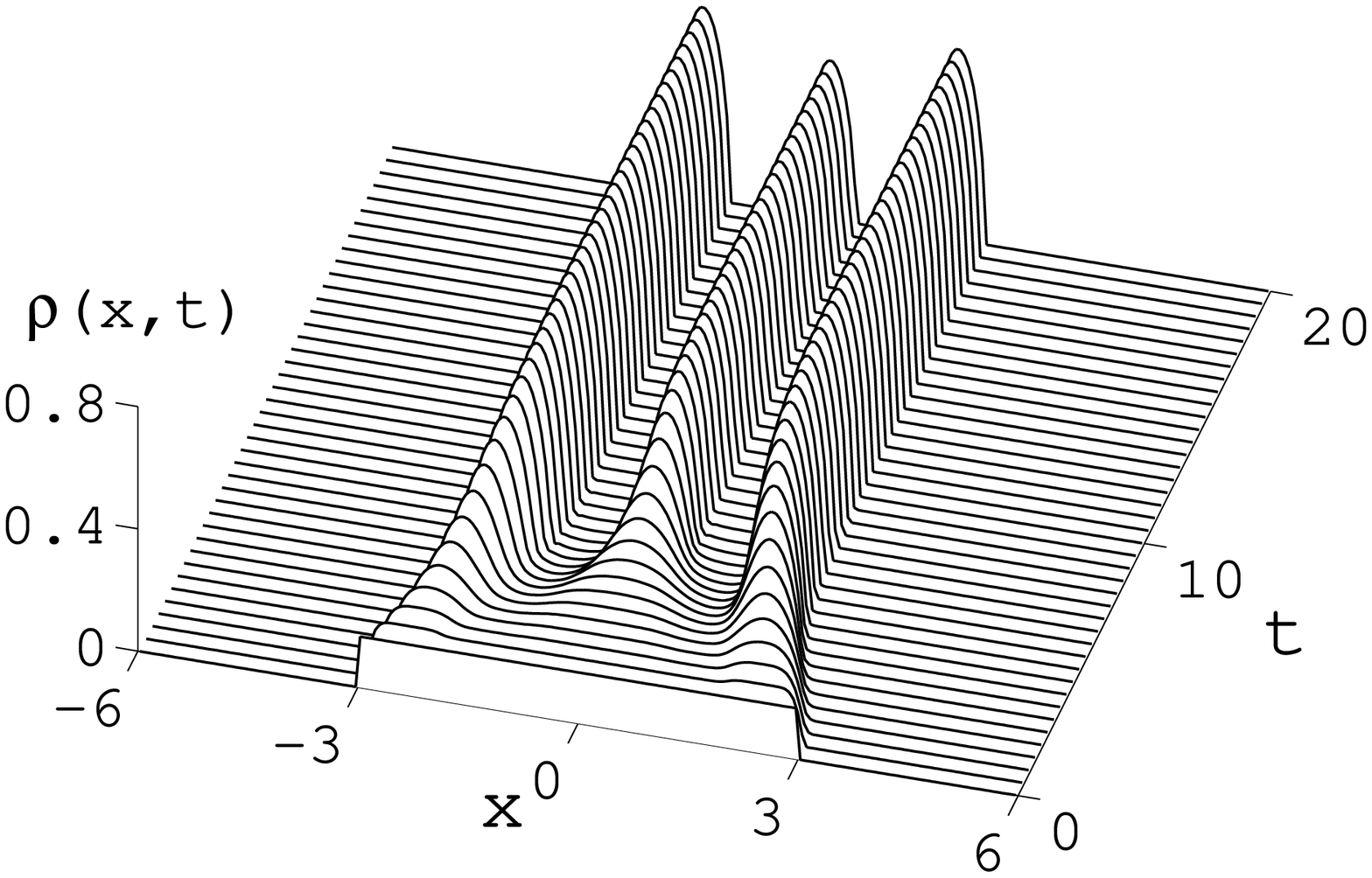}
\end{center}    
\caption{The dynamics of~\eqref{eq:nldiff} starting with the initial 
data $\rho_0=\chi_{[-2,2]}$ and $\rho_0=\chi_{[-3,3]}$.}
\label{fig:nldiff}
\end{figure}

The decay of the energy for these two cases is shown in Figure~\ref{fig:aggdiffentropy}{\bf (a)}. After the initial 
transient disappears, the energy decreases significantly at later times only when the topology changes, i.e. the merge of disconnected 
components. Although there is a steady state with one single component with all the mass, the three-bump
solutions with $\rho_0(x) =\chi_{[-3,3]}$ seems to be the correct final stable steady state. This can be confirmed from Figure~\ref{fig:aggdiffentropy}{\bf (b)}, as $\xi$ is a constant on each connected component of the support. 

This example shows a very interesting effect in this equation, which is the appearance in the long time asymptotics of steady states with 
disconnected support. It should be observed that each bump is at distance larger than 1 from the other bumps, and thus the interaction force 
exerted between them is zero. This together with the finite speed of propagation of the degenerate diffusion are the reasons why the steady 
state with the total mass and connected support is not achieved in the long time asymptotics. This fact is related to the existence of local 
minimizers of the functional \eqref{fe} in certain weak topologies, infinity Wasserstein distance, not allowing for large perturbations of the 
support, see \cite[Section 5]{BCLR2} and \cite{FR211} for related questions. 

\begin{figure}[hpt]
\begin{center}
\includegraphics[totalheight=0.25\textheight]{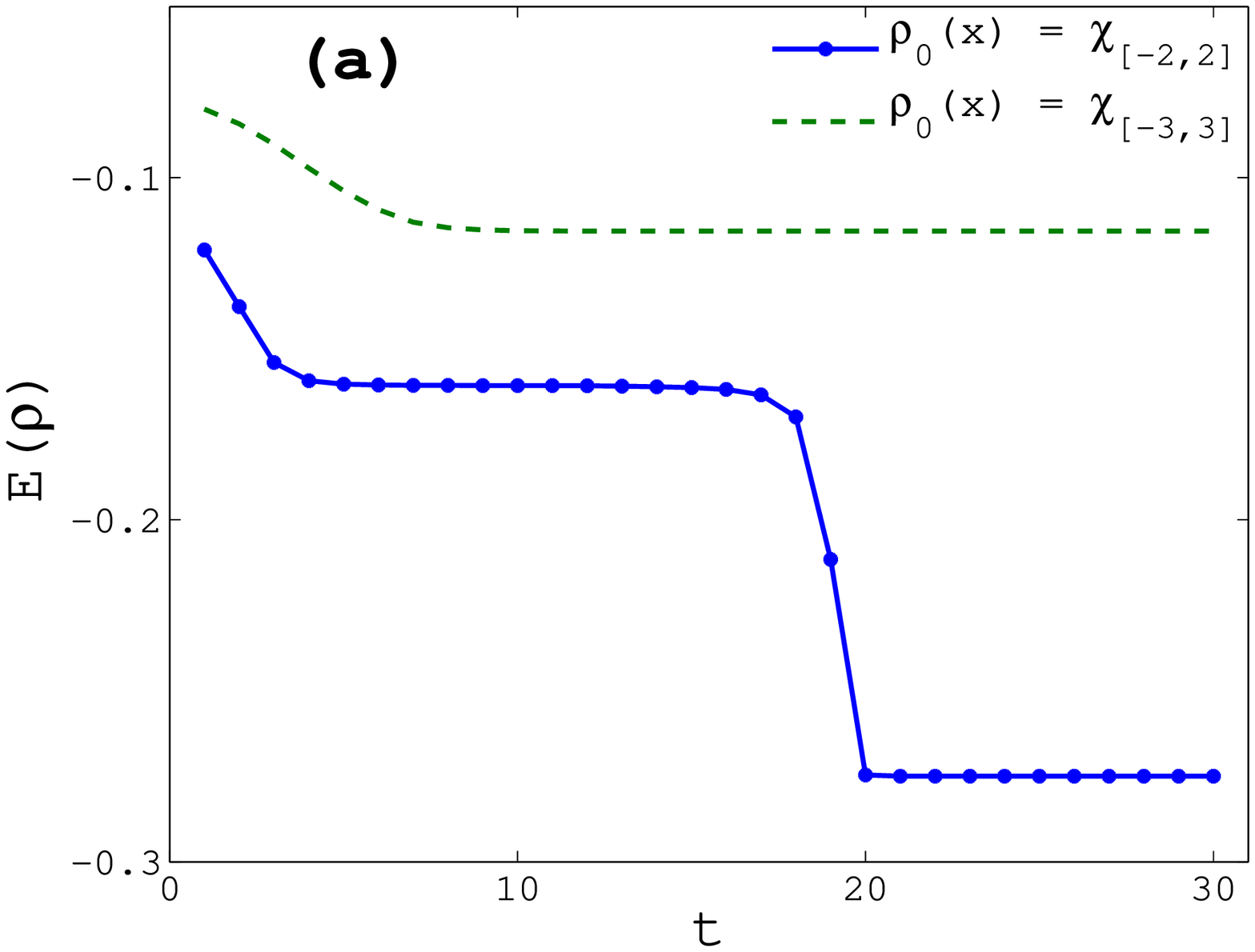}
$~~~$
\includegraphics[totalheight=0.25\textheight]{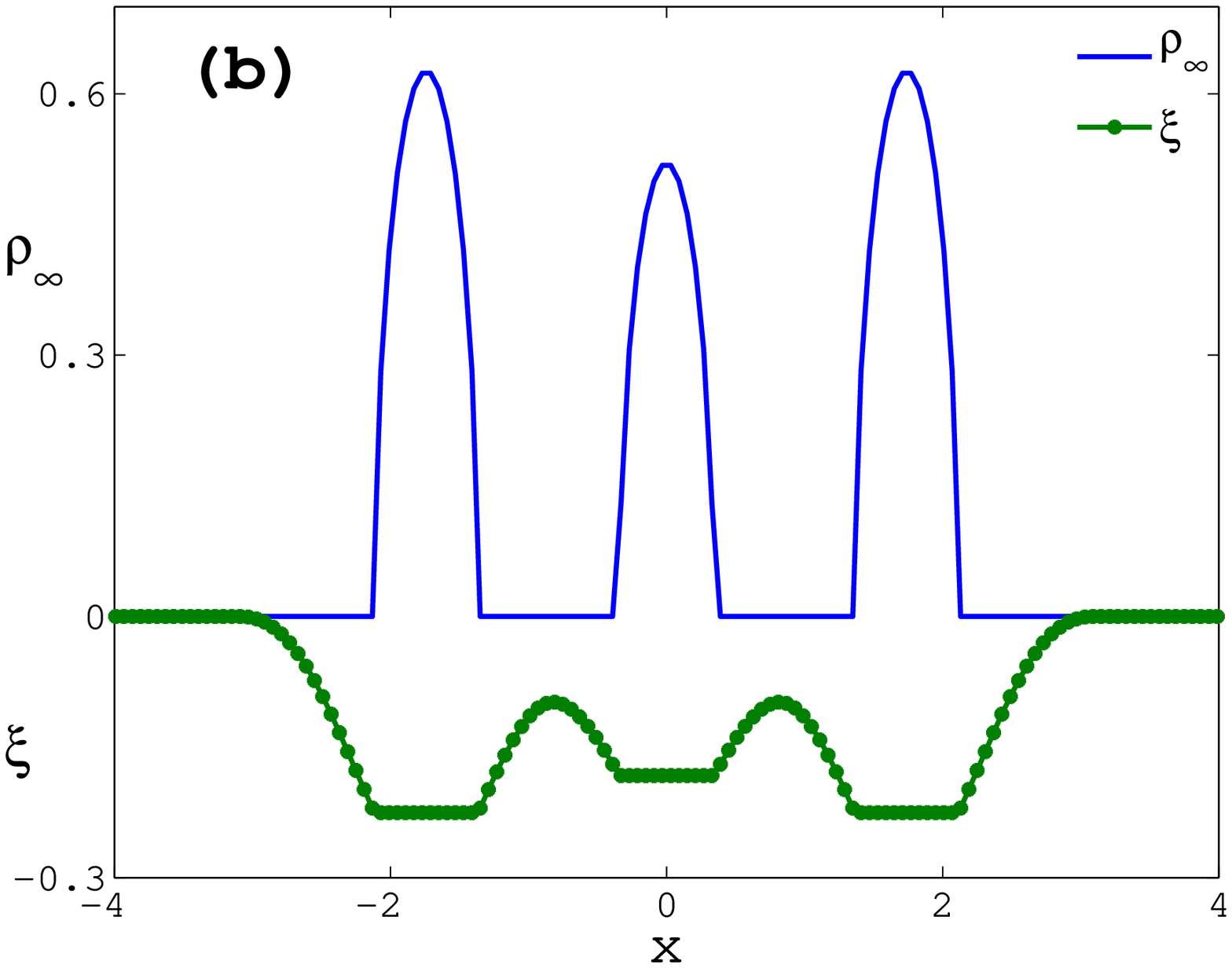}
\end{center}
\caption{{\bf (a)} The decay of the entropy of the equation~\eqref{eq:nldiff} 
with initial data $\rho_0(x)=\chi_{[-2,2]}$ and $\rho_0(x)=\chi_{[-3,3]}$. After an initial transient behavior, there is a significant decrease in the entropy only when the topology of the solution changes. {\bf (b)} The final steady state of~\eqref{eq:nldiff} with initial data $\rho_0(x)=\chi_{[-3,3]}$ and the corresponding $\xi$. Here $\xi$ assumes different constant values on different connected components of the support.}
\label{fig:aggdiffentropy}
\end{figure}

For other non-compactly supported kernels like $W(x)=-\frac{1}{2}e^{-|x|}$ or the Gaussian as in Example \ref{ex32}, there is a unique steady state with 
one single connected component in its support, though it exhibits the same slow-fast behavior in its convergence in time as shown in Figure 
\ref{fig:aggdiff1d_prof2}. This metastability and other decaying solutions when $m<2$ are discussed in more details in~\cite{BFH}.
\end{example}

\begin{example}[{\bf Nonlinear diffusion with double well external potential}]\label{ex34}
 In this example, we elaborate more on stationary states which are 
not global minimizers of the total energy. More precisely, we consider nonlinear diffusion equation for particles under an external double-well potential of the 
form
\begin{equation}\label{eq:doublewell}
\rho_t= \big( \rho (\nu\rho^{m-1}+V)_x\big)_x,\quad V(x)=\frac{x^4}{4}-\frac{x^2}{2}.
\end{equation}
 Actually, the steady states of \eqref{eq:doublewell} are of the form 
$$
\rho_\infty(x)=\left(\frac{C(x)-V(x)}{\nu}\right)_+^{\frac1{m-1}}
$$
with $C(x)$ piecewise constant possibly different in each connected component of the support.  

We run the computation with $\nu=1$, $m=2$ and initial data of the form
\begin{equation}\label{eq:dbwellinit}
\rho_0(x)=\frac{M}{\sqrt{2\pi\sigma^2}}e^{-\frac{(x-x_c)^2}{2\sigma^2}},\quad M=0.1,\ \sigma^2=0.2,
\end{equation}
corresponding to the symmetric ($x_c=0$) and asymmetric ($x_c=0.2$) cases, respectively. It is obvious that for small mass, we can get 
infinitely many stationary states with two connected components in its support by perturbing the value of $C$ defining a symmetric steady 
state. Actually, each of them has a non zero basin of attraction depending on the distribution of mass initially as shown in Figure 
\ref{fig:doublewell}{\bf (b)}.  While the global minimizer of the free energy is the symmetric steady state, the non symmetric ones are local 
minimizers in the infinity Wassertein distance or informally for small perturbations in the sense of its support. It is interesting to observe that 
even if the long time asymptotics is different for each initial data, the rate of convergence to stabilization seems uniformly 2, see Figure \ref{fig:doublewell}{\bf (a)}.

\begin{figure}[hpt]
\begin{center}
\includegraphics[totalheight=0.25\textheight]{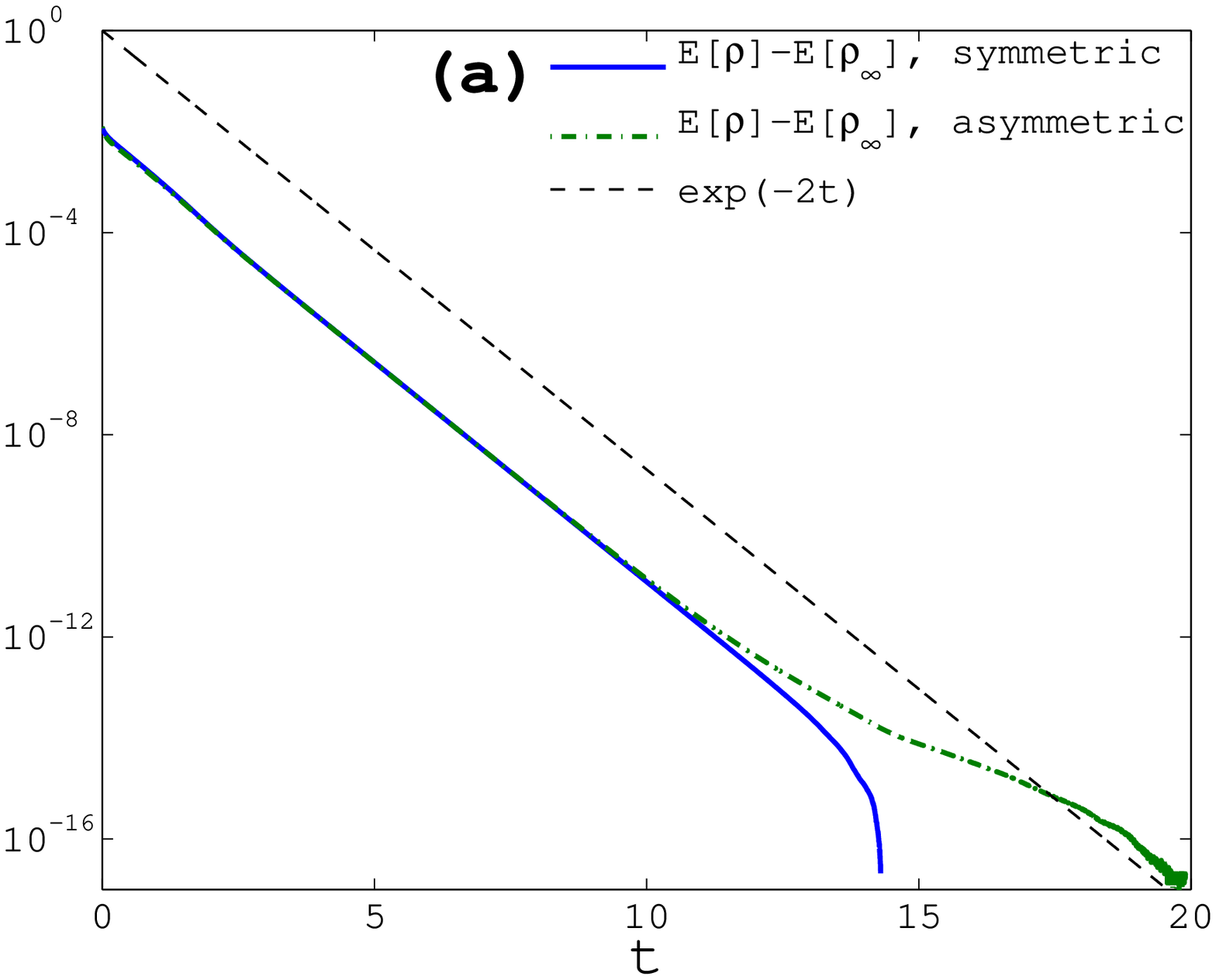}
$~~~$
\includegraphics[totalheight=0.25\textheight]{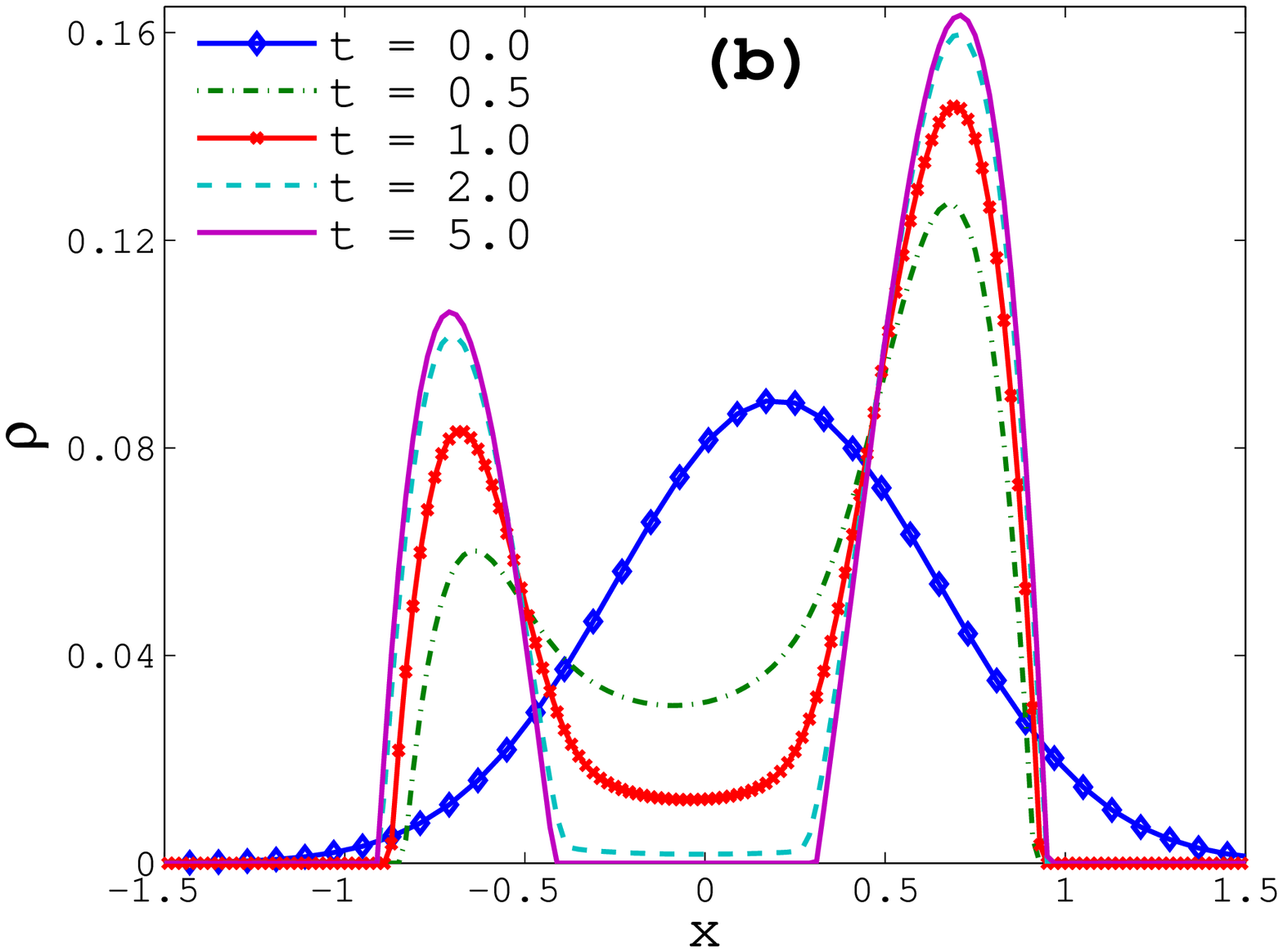}
\end{center}
\caption{ {\bf (a)} The decay of the entropy of the equation~\eqref{eq:doublewell} with initial data~\eqref{eq:dbwellinit}, for the symmetric ($x_c=0$)
and asymmetric ($x_c=0.2$) cases, respectively.  A uniform rate of convergence of order 2 is observed towards the stationary states. {\bf (b)} The evolution of the asymmetric initial data ($x_c=0.2$) towards the corresponding asymmetric stationary state.}
\label{fig:doublewell}
\end{figure}
\end{example}

\subsection{Generalized Keller-Segel model}\label{sec33}
Another related diffusion equation with nonlocal attraction is the generalized Keller-Segel model,
\begin{equation}
 \rho_t = \nabla\cdot\big(\rho \nabla(\nu \rho^{m-1}+W\ast\rho)\big),
\label{gks}
\end{equation}
with the kernel $W(\mathbf{x}) = |\mathbf{x}|^{\alpha}/\alpha$ with $-d<\alpha$ or the convention $W(\mathbf{x}) = \ln |\mathbf{x}|$ for $\alpha=0$. The bound from below in 
$\alpha$ due to local integrability of the kernel $W$.
When $\alpha = 2-d$, $W$ is the Newtonian potential in $\mathbb{R}^d$, and the equation reduces to the Keller-Segel model for chemotaxis 
with nonlinear diffusion:
\begin{equation}\label{eq:PKS}
    \rho_t = \nabla\cdot\big(\rho \nabla(\nu \rho^{m-1}-c)\big),\quad 
    -\Delta c = \rho.
\end{equation}
Compared with Example \ref{ex32} where the interaction potential $W$ is integrable, the long tail for $W(\mathbf{x}) = |\mathbf{x}|^{\alpha}/\alpha$ has non-trivial consequences. In 
certain parameter regimes, the solution can even blow up in finite time with smooth initial data. To clarify the different regimes, we can easily 
evaluate the balance between the attraction due to the nonlocal kernel $W$ and the repulsion due to diffusion by scaling arguments. In fact, taking 
the corresponding energy functional \eqref{fe} and checking the scaling under dilations of each term, we can find three different regimes:
\begin{itemize}
\addtolength{\itemsep}{-0.5\baselineskip}
\item Diffusion-dominated regime: $m>(d-\alpha)/d$. Here, the intuition is that solutions exist globally in time and the aggregation effect only 
shows in the long-time behavior where we observe nontrivial compactly supported stationary states.
\item Balanced regime: $m=(d-\alpha)/d$. Here the mass of the system is the critical quantity. There is a  
critical mass, separating the diffusive behavior from the blow-up behavior.
\item Aggregation-dominated regime: $m<(d-\alpha)/d$. Blow-up  and diffusive behavior coexist for all values of the mass 
depending on the initial concentration.
\end{itemize}

\begin{figure}[htb]
    \begin{center}
        \includegraphics[totalheight=0.21\textheight]{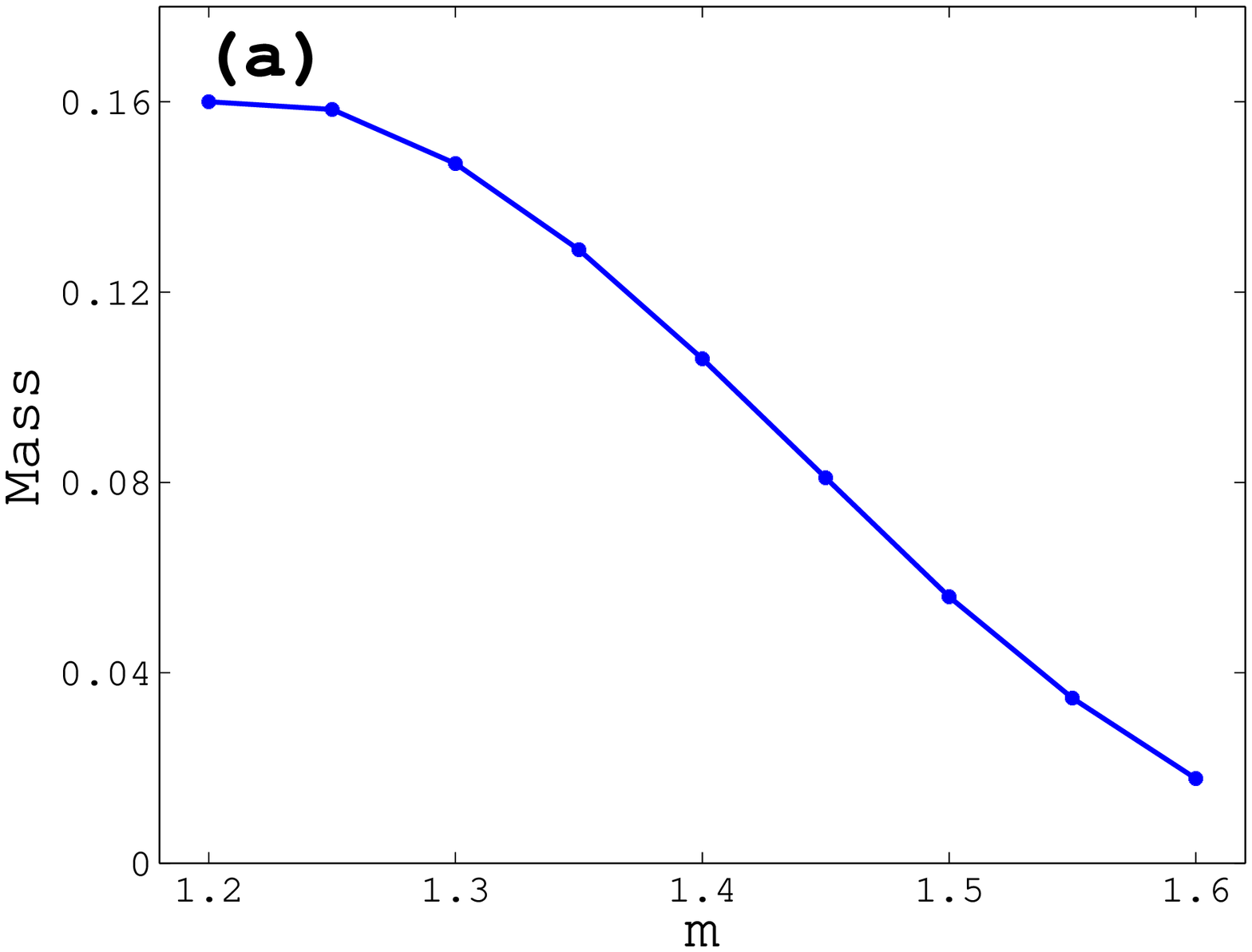}$~~$    
         \includegraphics[totalheight=0.21\textheight]{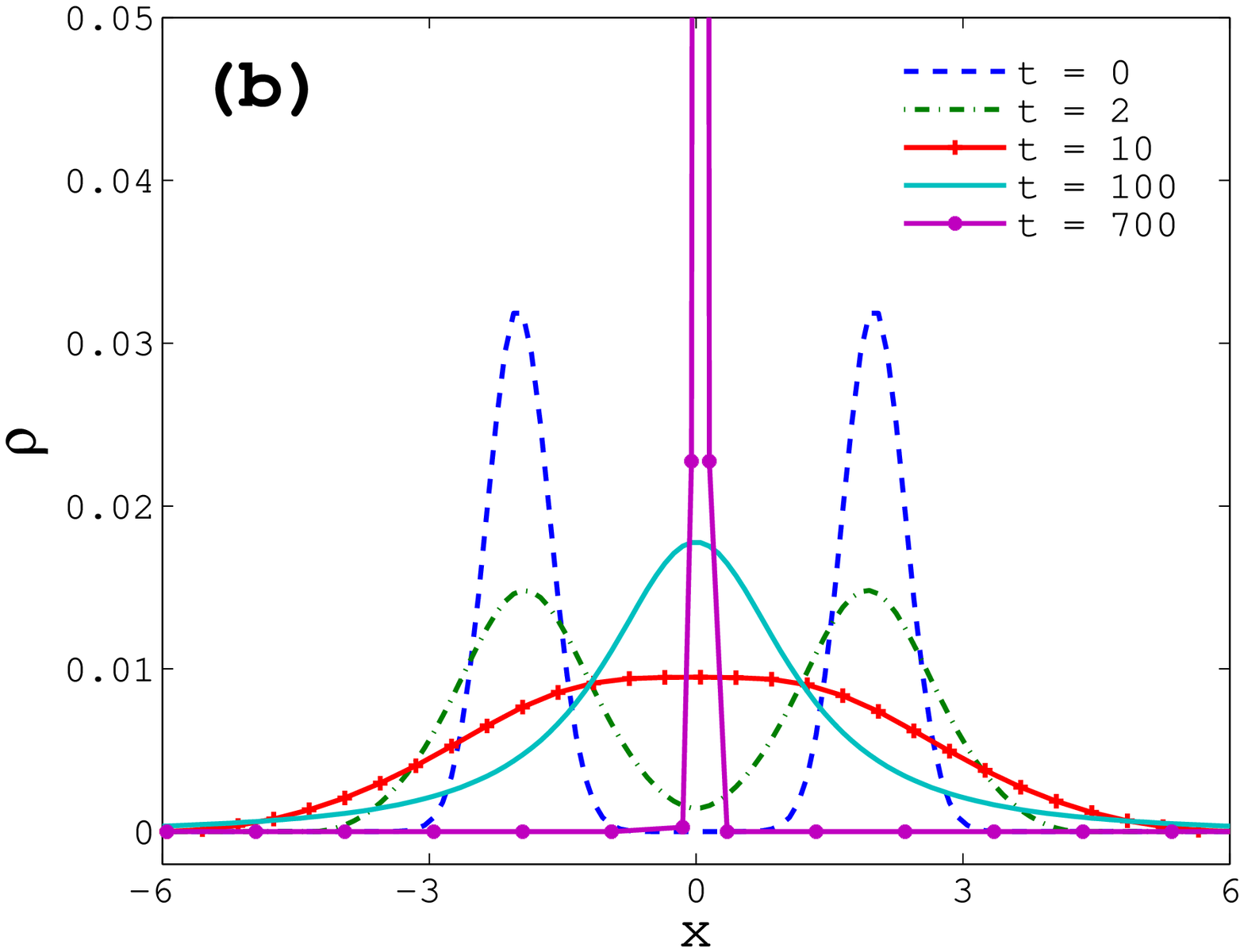}
         \includegraphics[totalheight=0.21\textheight]{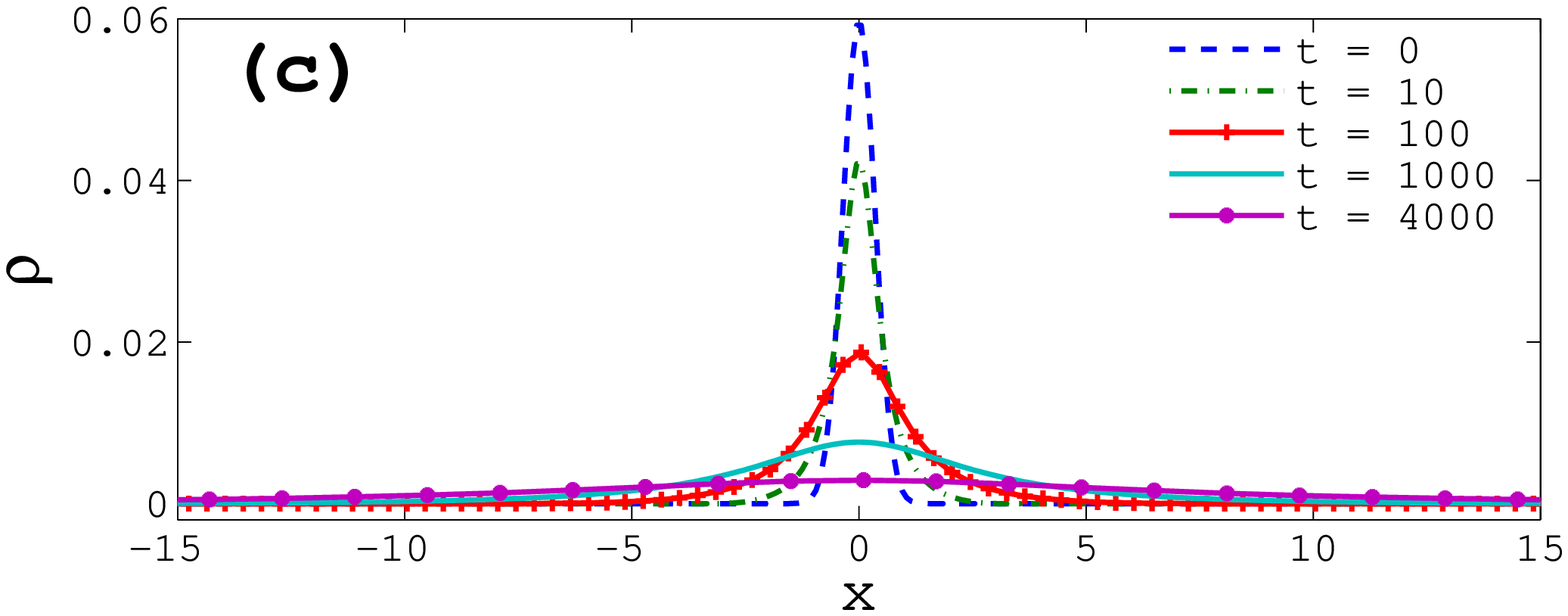}
    \end{center}
    \caption{{\bf (a)} The critical mass $M_c$ when $m+\alpha=1$, $\nu=1$
 for different  exponents $m$. {\bf (b)} The blowing up solution for $m=1.5$, $\alpha=-0.5$ and $\nu=1$ with initial data 
    $\rho_0(x) = M(e^{-4(x+2)^2}+e^{-4(x-2)^2})/\sqrt{\pi}$, where  the 
    total mass $M = 0.057 > M_c \approx 0.055$. {\bf (c)} The decaying solution for $m=1.5$, $\alpha=-0.5$ and $\nu=1$ with initial data $\rho_0(x) = Me^{-x^2}/\sqrt{\pi}$, where $M=0.53<M_c=0.55$.} 
    \label{fig:critmass}
\end{figure}

The classical 2-D Keller-Segel system corresponds to $m=1$, $\alpha=0$, see \cite{V,CC,BCM,BCC} and the references therein 
for the different behaviors. The nonlinear diffusion model for the balanced case with the Newtonian potential in $d\geq 3$ was studied in detail 
in~\cite{BCL}. Finite time blow-up solutions for general kernel $W(\mathbf{x}) = |\mathbf{x}|^{\alpha}/\alpha$ in the aggregation-dominated regime were also investigated, combined with
numerical simulations~\cite{YaoBer}.

\begin{example}[{\bf Generalized Keller-Segel model in the balanced regime}]\label{ex35}
Let us start with the 1-D example when $m+\alpha=1$  corresponding to the balanced case. Here, the behavior of the dynamics depends on the total conserved mass. The solutions blow up if the mass is 
greater than the threshold $M_c$ and otherwise the solutions decay to zero. This threshold mass can be determined by solving the equation 
with different initial conditions and is shown in Figure~\ref{fig:critmass}{\bf (a)} for different values of $m$. For example, when $m=1.5$ and $\alpha=1-m=-0.5$, the 
threshold mass $M_c$ is about $0.055$. If the initial data has a larger mass as in Figure~\ref{fig:critmass}{\bf (b)}, the solution blows up. Since the numerical method is conservative, the density concentrates inside one cell instead of being infinity. 
Otherwise, if the initial data has a smaller mass as in  Figure~\ref{fig:critmass}{\bf (c)}, the solution decays to zero.

We have also checked the self-similar behavior for subcritical mass cases $(M<M_c)$ in the sense of solving \eqref{gks} 
with $V(x)=|x|^2/2$. That is in the similarity variables, the solution of $\rho_t = \nabla\cdot\big(\rho \nabla(\nu \rho^{m-1}+W\ast\rho+|x|^2/2)\big)$ converges to the self-similar profile. 
The decay rate in time is computed for several  subcritical masses and is shown in Figure~\ref{fig:subcritmass}{\bf (a)}, illustrating that this rate is independent of the mass and is exactly $O(e^{-2t})$ as proven in the classical 2-D 
Keller-Segel model in \cite{CD}. We also observe in Figure~\ref{fig:subcritmass}{\bf (b)} how the self-similar profiles become concentrated as a Dirac Delta at the origin as $M\to M_c$.
\begin{figure}[htpb]
    \begin{center}
        \includegraphics[totalheight=0.25\textheight]{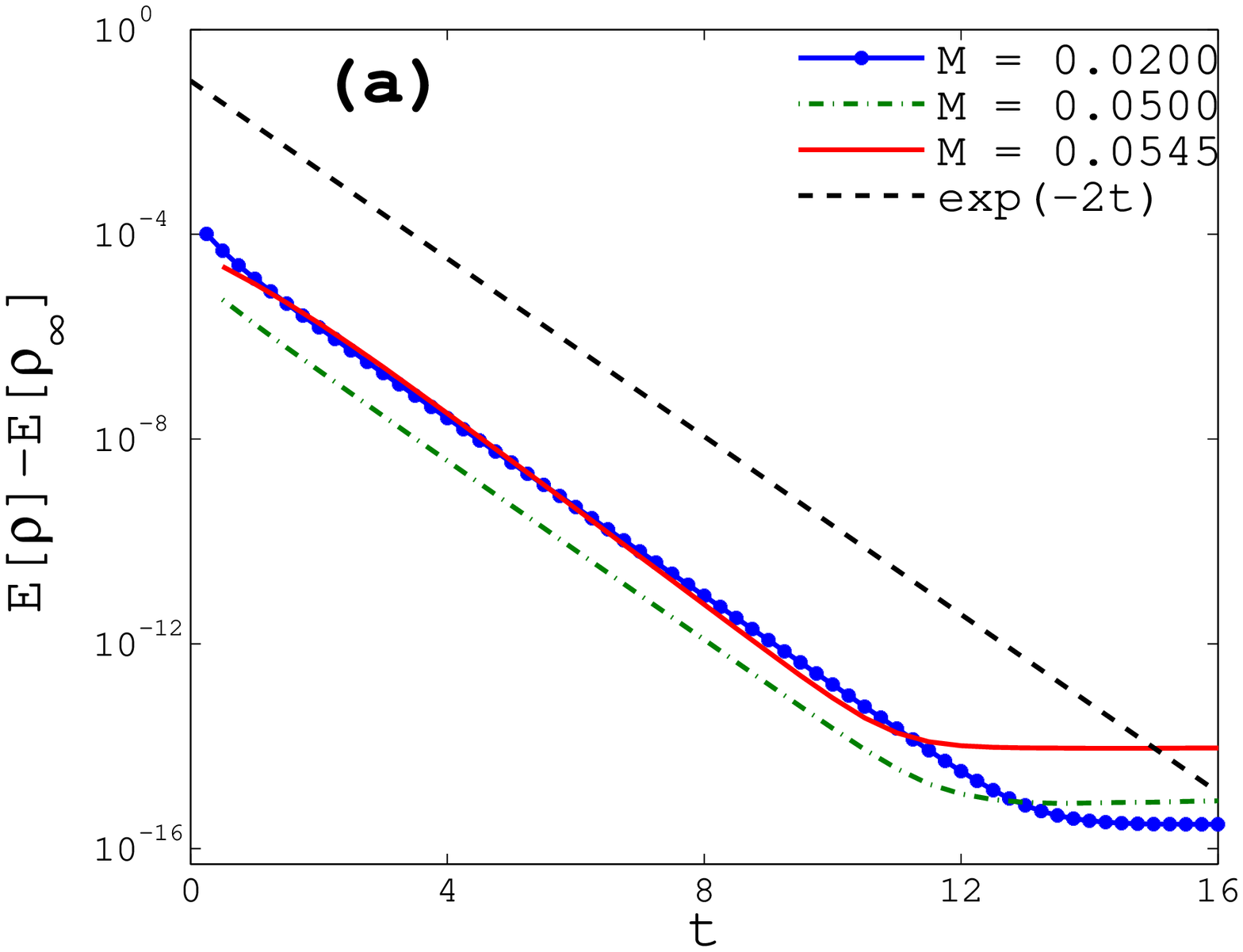}
$~~$    
        \includegraphics[totalheight=0.25\textheight]{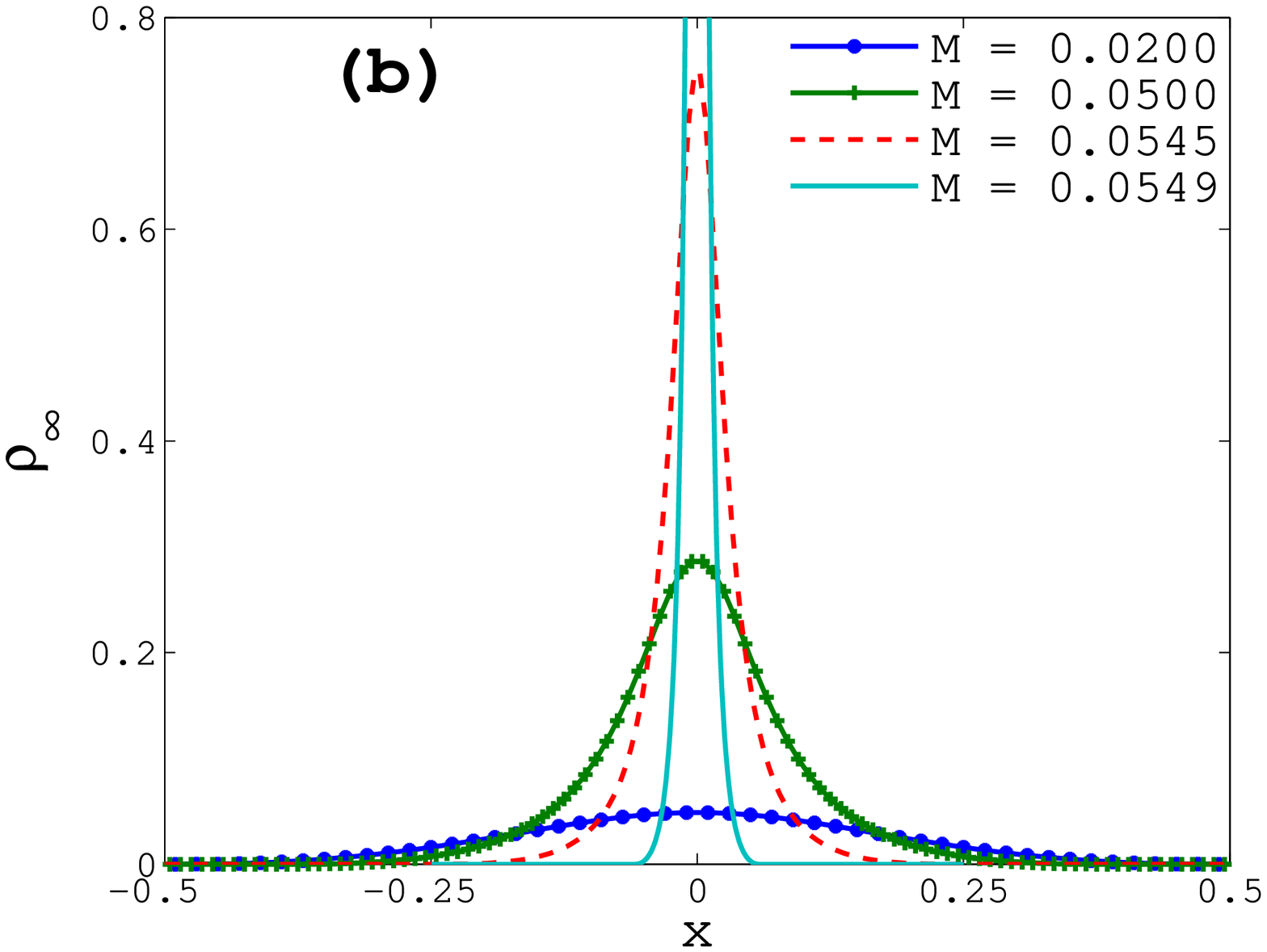}
    \end{center}
    \caption{{\bf (a)} The uniformly exponential decay towards equilibrium (in similarity variables) for subcritical mass in self-similar variables when $m=1.5$, $\alpha=-0.5$, $\nu=1$ 
for different values of the mass $M<M_c$. {\bf (b)} The equilibrium profiles for different $M<M_c$.} 
    \label{fig:subcritmass}
\end{figure}
\end{example}

\begin{example}[{\bf Generalized Keller-Segel model in the other regimes}]\label{ex36}
The general behaviors of solutions to the 1-D version of ~\eqref{gks} in other 
parameter regimes are also known to some extent. When $m>1-\alpha$ corresponding to the diffusion-dominated regime, a compact steady 
state is always expected, which is the global minimizer of the energy~\eqref{fe} as in \cite{St}. If the nonlinearity of the diffusion is increased to 
be $m=1.6$ with the same total mass $(=0.057)$ and the exponent $\alpha=-0.5$, the solution converges to a steady state as in Figure~\ref{fig:KSsupexp}
 instead of blowing up as in Figure~\ref{fig:critmass}{\bf (b)}.  When $\alpha+m<1$ corresponding to the aggregation-dominated regime, the small initial data decays to zero while large initial data blows up in finite time (see Figure~\ref{fig:KSsubexp}). The size of the initial data determining the distinct behaviors 
is usually measured in a norm different from $L^1$ (the conserved mass), and
no critical value in this norm as in the case $m+\alpha=1$ is expected.

\begin{figure}[htp]
\begin{center}
\includegraphics[totalheight=0.25\textheight]{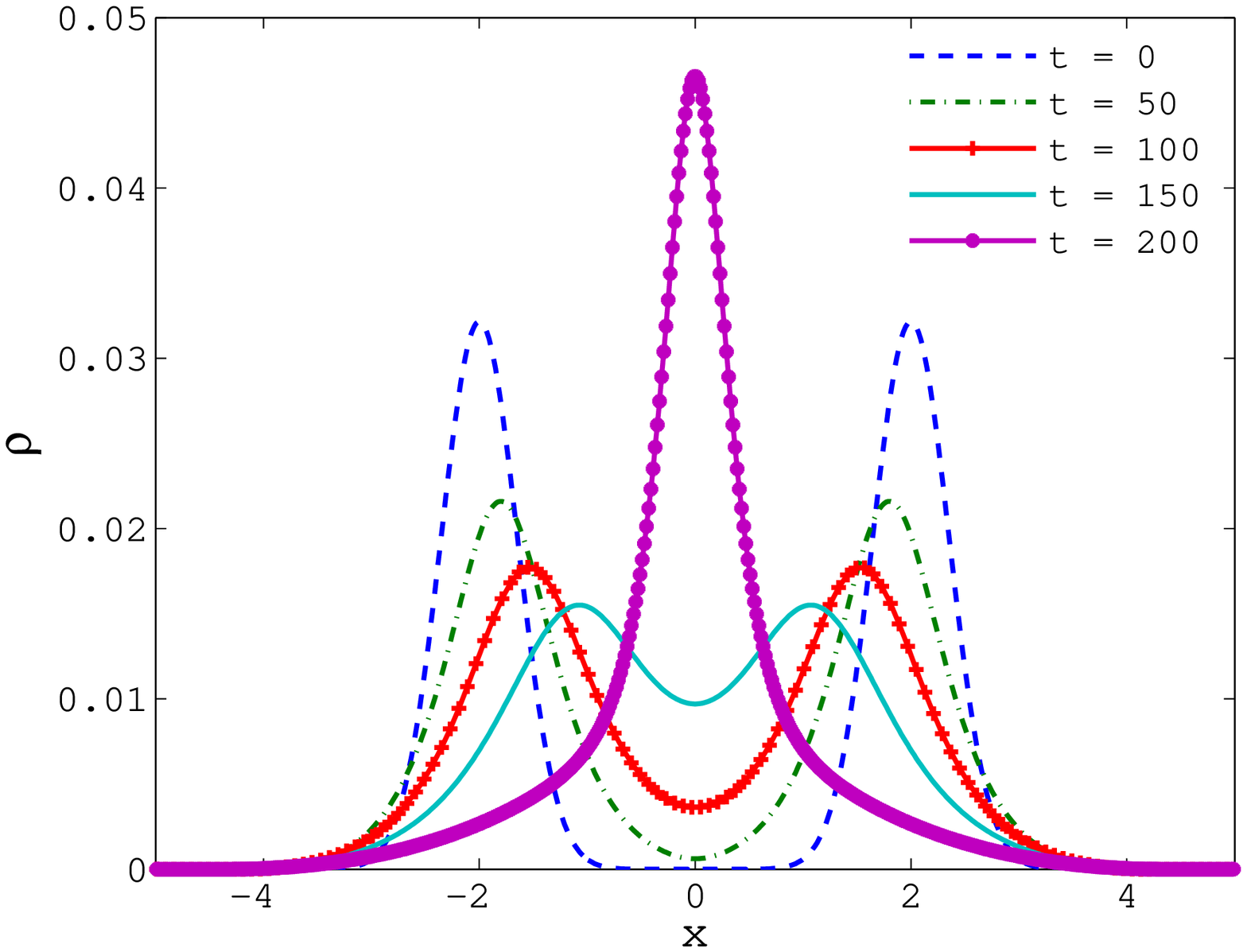}
$~~$
\includegraphics[totalheight=0.25\textheight]{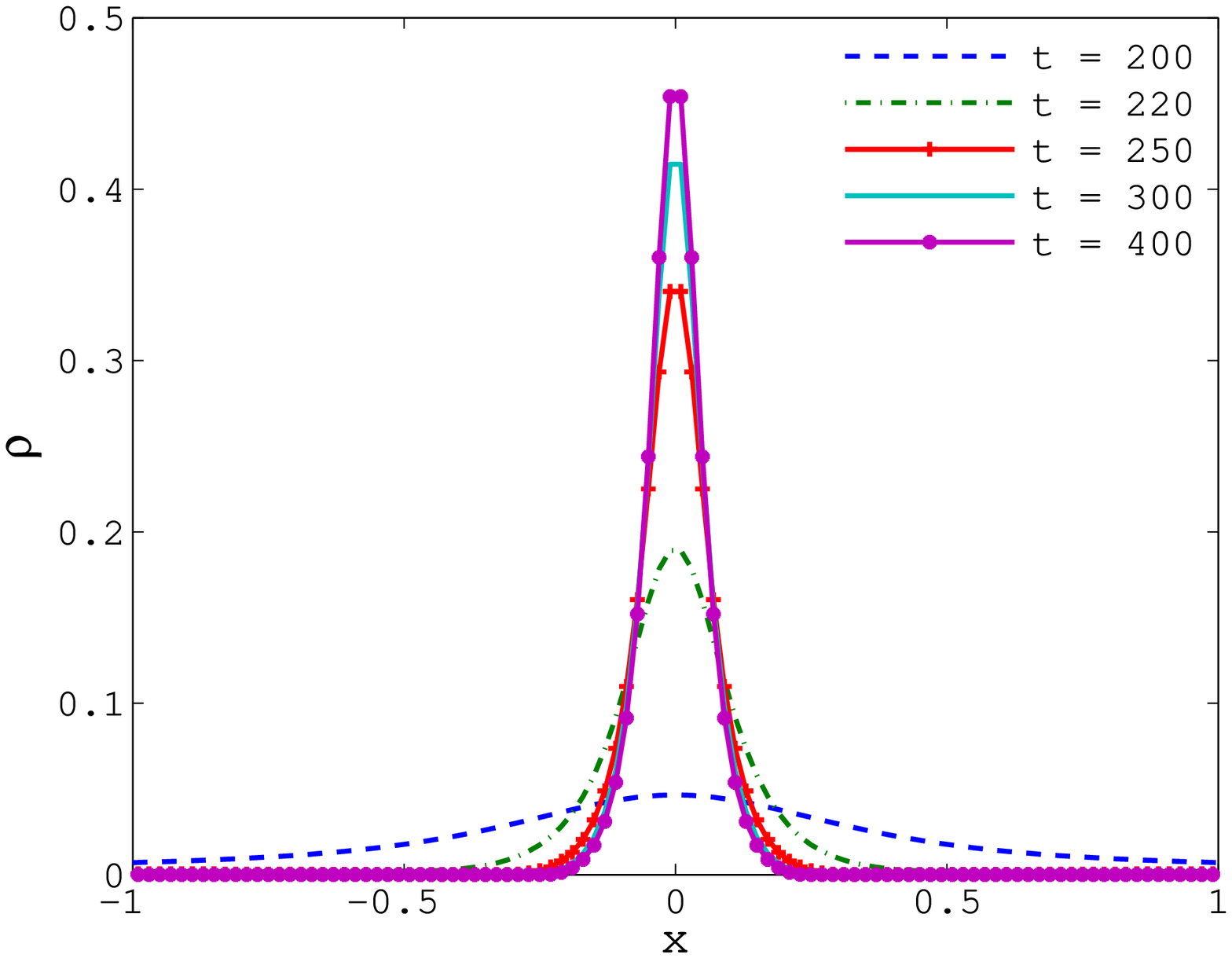}
\end{center}
\caption{ The evolution of the generalized Keller-Segel equation in the diffusion dominated regime ($m=1.6$, $\alpha = -0.5$) with
$\nu=1.0$. The initial condition $\rho_0(x) = M(e^{-4(x+2)^2}+e^{-4(x-2)^2})/\sqrt{\pi}$ ($M=0.057$) is the same as that in Figure~\ref{fig:critmass} {\bf (b)}.}
\label{fig:KSsupexp}
\end{figure}

\begin{figure}[htp]
\begin{center}
\includegraphics[totalheight=0.25\textheight]{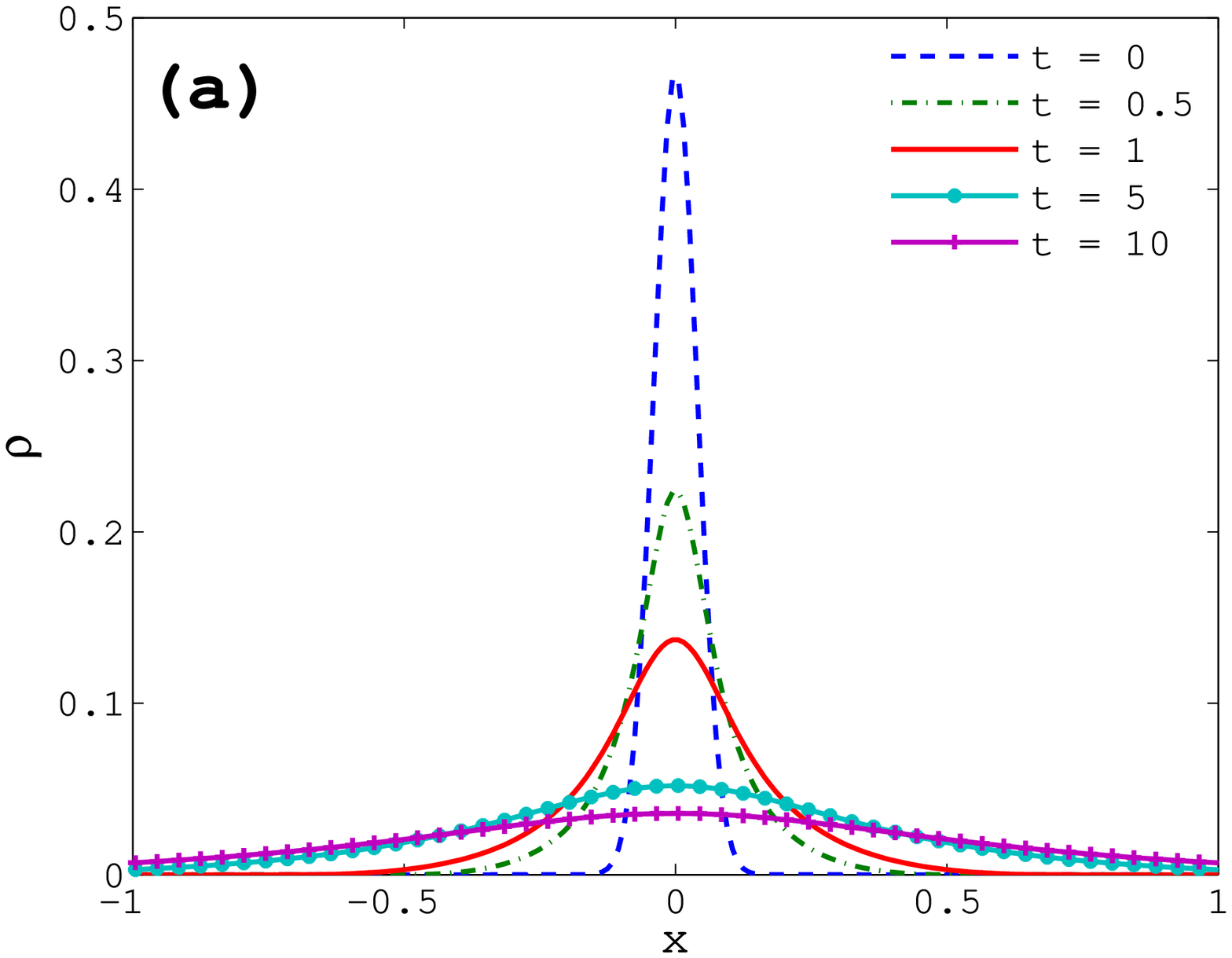}
$~~$
\includegraphics[totalheight=0.25\textheight]{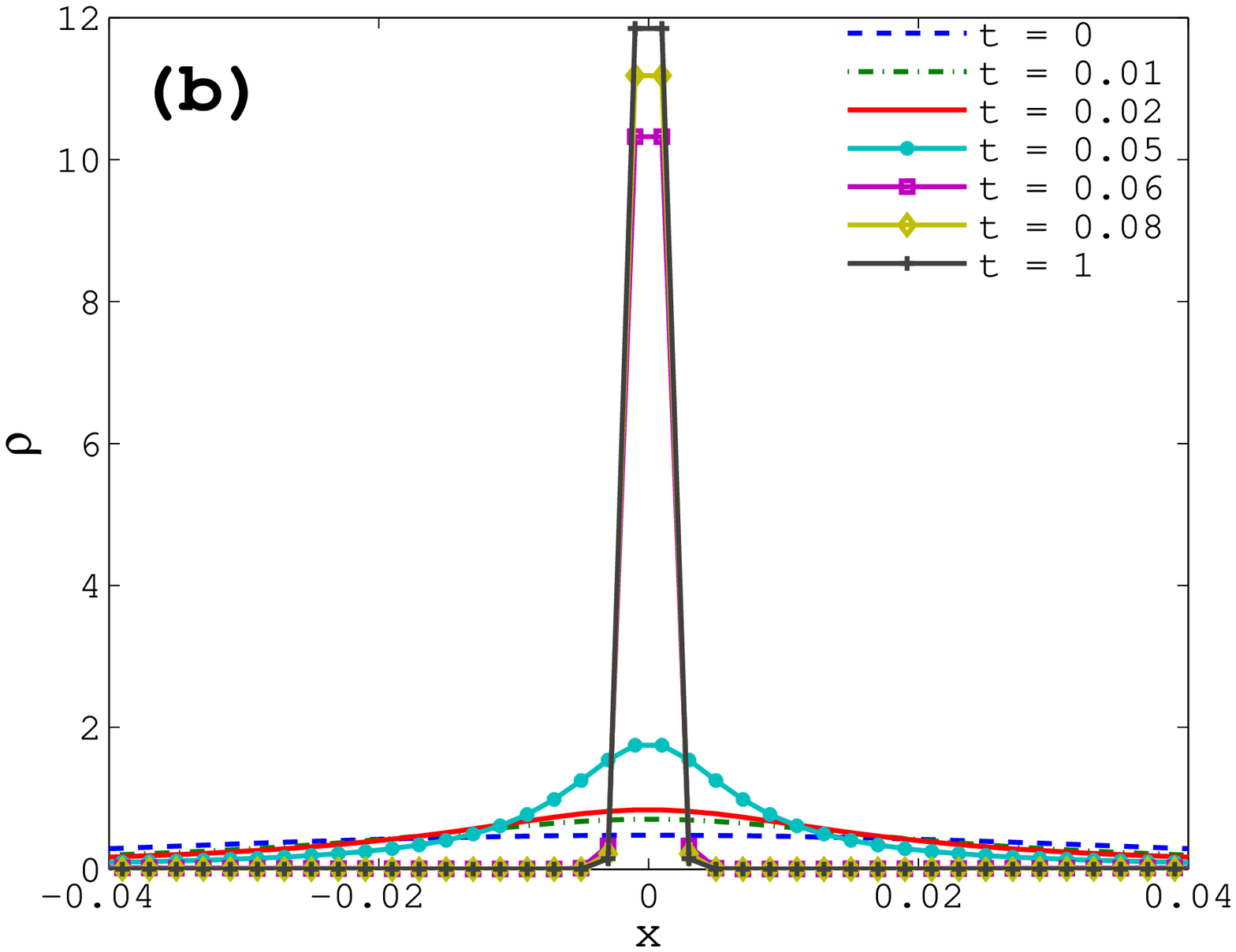}
\end{center}
\caption{ The evolution of the generalized Keller-Segel equation in the aggregation-dominated regime ($m=1.6$, $\alpha = -0.5$) with $\nu=1.0$.
The initial condition is $\rho_0(x) = M(e^{-4(x+2)^2}+e^{-4(x-2)^2})/\sqrt{\pi}$, with $M=0.047$ for decaying solution in {\bf (a)} 
and $M=0.048$ for blowup solutions in {\bf (b)}.}
\label{fig:KSsubexp}
\end{figure}

\end{example}

\subsection{Aggregation equation with repulsive-attractive kernels}
\label{sec:attrep}
In the absence of diffusion from $H(\rho)$ or confinement from
 $V$, steady states of the general equation~\eqref{1.1}
are still expected when the kernel $W$ incorporates both 
short range repulsion and long range attraction. This type of kernels
arises in the continuum formulation of moving flocks of self-propelled
particles~\cite{PhysRevE.63.017101,d2006self}, and the popular ones are the
Morse potential 
\[
    W(x) = Ce^{-|x|/\ell}-e^{-|x|}, \quad C>1,\ell<1
\]
and the power-law type
\[
    W(x) = \frac{|x|^a}{a} - \frac{|x|^b}{b},\quad a>b,
\]
with the convention that $|x|^0/0 = \ln |x|$ below.

\begin{example}[{\bf Quadratic attractive and Newtonian repulsive kernels}]\label{ex37}
The regularity of the solution  depends on the singularity of the repulsion force.
If this force is small at short distance (or equivalently $b$ is relatively large),
the solution can concentrate at a lower dimension subset, while 
more singular forces lead to smooth steady states except possible discontinuities near 
the boundary~\cite{BCLR2}. The case $a=2$ and $b=0$ is shown in Example \ref{ex31}, whose steady state is a semi-circle~\cite{MR1485778, CFP}, while the case $a=2$ and $b=1$ leads to a steady state
which is a constant supported on an interval~\cite{FR11,FHK}. 

We remind that the discrete convolution for the velocity field in \eqref{tech1} is discretized using the coefficients $W_{j-i}$, chosen as approximations of the local integral 
\begin{equation}\label{tech2}
W_{j-i}=\frac{1}{\Delta x} \int_{C_i} W(x_j-s)ds.
\end{equation}
In the case of smooth kernels $(b>0)$, we can either use the mid-point rule or a direct computation of the integral if available. We show the numerically computed stationary state with both options in Figure~\ref{fig:quadnewton} {\bf (a)} and {\bf (b)} respectively.
As one can observe, the first choice is oscillation free while the second choice with exact integrals shows an overshoot of the density near the boundary of the support. The difference between the two cases can be explained by carefully writing down the characterization of the discrete stationary states based on the discrete entropy inequality in Theorem \ref{dd1}. The mid-point rule performs better due to its symmetry that induces some numerical diffusion. 

\begin{figure}[ht]
\begin{center}
\includegraphics[totalheight=0.24\textheight]{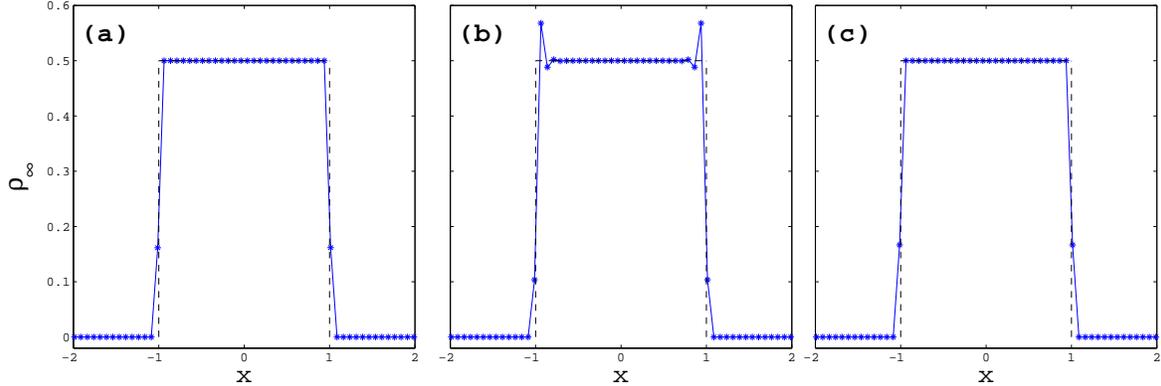}
\end{center}
\caption{The steady states computed with: {\bf (a)} mid-point quadrature rule for \eqref{tech2}; {\bf (b)} exact computation of $W_{j-i}$ in \eqref{tech2}; {\bf (c)} Same as {\bf (b)} but adding small nonlinear diffusion.} 
\label{fig:quadnewton}
\end{figure}

In case we would be dealing with singular kernels, we cannot use simple quadrature formulas like middle-point rule but rather we need to implement either quadrature formulas for singular integrals or perform exact evaluations of the integrals in \eqref{tech2}. To avoid the oscillations as in Figure~\ref{fig:quadnewton}{\bf (b)}, we added a small nonlinear diffusion term, i.e., $\rho_t = \big(\rho (\epsilon \rho + W*\rho)_x)_x$. Here quadratic nonlinear
diffusion is used, respecting the same nonlinearity and scaling as
in the original equation $\rho_t = (\rho (W*\rho)_x)_x$.  Numerical experiments as in Figure~\ref{fig:quadnewton}{\bf (c)} indicate that $\epsilon = 
0.25(\Delta x)^2$ is close to optimal, in the sense 
that $\epsilon$ is just large enough to prevent the overshoot. This
near optimal diffusion coefficient has been further confirmed by numerical experiments with different $\Delta x$. 

\begin{figure}[htp]
\begin{center}
\includegraphics[totalheight=0.24\textheight]{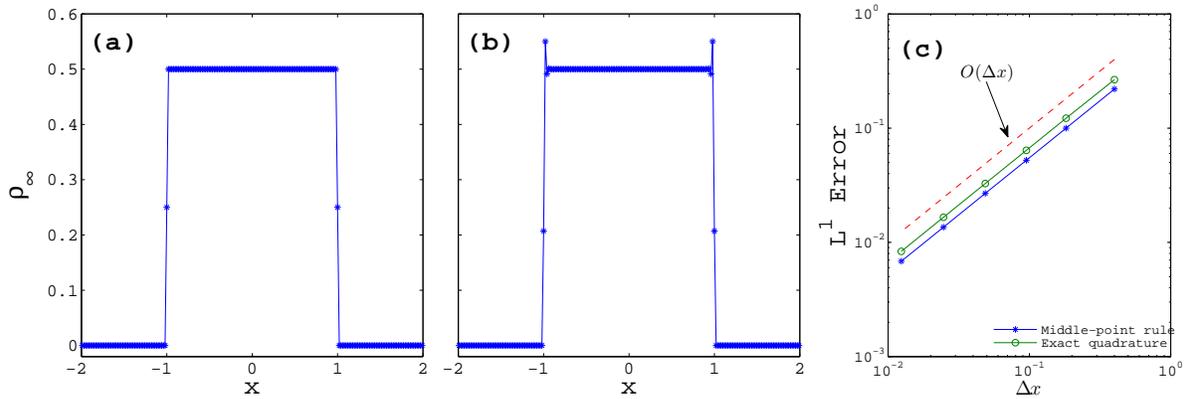}
\end{center}
\caption{The steady states computed with on a finer mesh with: {\bf (a)} mid-point rule for \eqref{tech2}; {\bf (b)} exact computation of $W_{j-i}$; {\bf (c)} the convergence of $L^1$ errors for both options.}
\label{fig:optimaldx}
\end{figure}
\end{example}

For the sake of clarity, we show in Figure~\ref{fig:optimaldx}{\bf (a)}-{\bf (b)} the steady-state solutions computed on a finer mesh for the same cases as in Figure~\ref{fig:quadnewton}{\bf (a)}-{\bf (b)} along with the $O(\Delta x)$ decay of $L^1$ errors for different grid sizes $\Delta x$ in Figure~\ref{fig:optimaldx}{\bf (c)}. The $L^\infty$ errors is almost constant and not decaying with mesh refinement.
They clearly indicate that the overshoot amplitude seen in Figures \ref{fig:quadnewton}{\bf (b)} and \ref{fig:optimaldx}{\bf (b)} is not reduced by mesh refinement and it needs the fix of small diffusion regularization. This will be further discussed in 2-D simulations below.

\subsection{Two-dimensional simulations}\label{sec34}

Now, let us illustrate the performance of the scheme in 2-D with some selected examples showcasing different numerical difficulties and interesting asymptotics. 

\begin{figure}[thp]
    \begin{center}
        \includegraphics[totalheight=0.55\textheight]{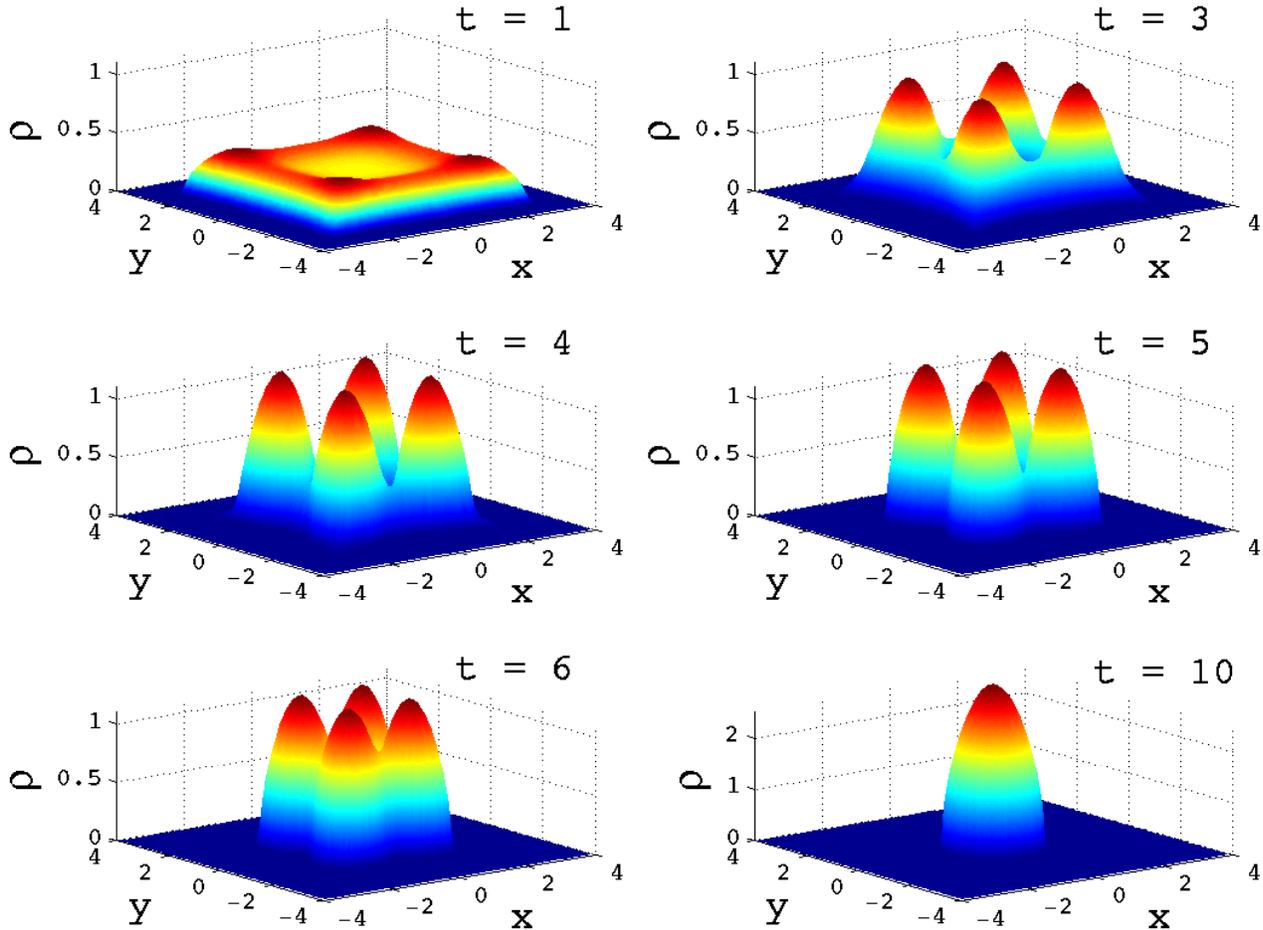}
    \end{center}
\caption{The evolution of the 2d aggregation equation with nonlinear diffusion
with $\nu = 0.1$, $m=3$, $W(\mathbf{x}) = \exp(-|\mathbf{x}|^2)/\pi$ and 
initial condition $\rho_0(\mathbf{x}) = \frac{1}{4}\chi_{[-3,3]\times[-3,3]}(\mathbf{x})$. 
The computational domain is $[-4,4]\times [-4,4]$, with grid size $\Delta x=\Delta y=0.1$ and 
time step $\Delta t=0.001$.}
\label{fig:2daggdiff}
\end{figure}

\begin{example}[{\bf Nonlinear diffusion with nonlocal attraction in 2-D}]\label{ex39}

For the equation with $H(\rho) = \frac{\nu}{m}\rho^{m}$, $W(\mathbf{x}) = -\exp(-|\mathbf{x}|^2)/\pi$ and $V\equiv 0$, the dynamics is similar to that in 1-D, being the result of the competition between the nonlinear diffusion $\nabla\cdot\big(\rho\nabla(\nu \rho^{m-1})\big)$ and the nonlocal attraction $\nabla\cdot\big(\rho\nabla W*\rho)\big)$. The evolution starting from the rescaled characteristic 
function supported on the square $[-3,3]\times [-3,3]$ is shown in 
Figure~\ref{fig:2daggdiff}. Because the interaction represented by the kernel $W(\mathbf{x})$ is nonzero for any $\mathbf{x}=(x,y)$, the final steady state
consists of one single component; however, four clumps are formed 
in the evolution, as the attraction dominates the relatively weak diffusion.

\begin{example}[{\bf Quadratic attractive and Newtonian repulsive kernel with small nonlinear diffusion}]\label{ex310} 
Similarly, overshoots may appear near the boundary of discontinuous solutions of $\rho_t = \nabla\cdot\big(\rho \nabla W*\rho\big)$ with repulsive-attractive kernels $W$. These overshoots can not be eliminated as easily as in one dimension, either by a careful choice of grid to align with the boundary or by a special numerical quadrature for $W_{i-j}$. However, stable solutions 
can be obtained by adding small nonlinear diffusion as in Example~\ref{ex37}. Therefore, we consider the equation
\[
 \rho_t = \nabla\cdot\big(\rho \nabla (\epsilon \rho + W*\rho)\big).
\]
For quadratic attractive and Newtonian repulsive kernel 
$W(\mathbf{x}) =|\mathbf{x}|^2/2 - \ln |\mathbf{x}|$, the steady states are shown in Figure~\ref{fig:quadlog2d}, without $(\epsilon=0)$ or with the diffusion. 
The near optimal coefficient $\epsilon$ is numerically shown to be close to 
$0.4((\Delta x)^2+(\Delta y)^2)$, exhibiting a similar mesh dependence as in Example~\ref{ex37}. Since $W$ is singular in this (and next) example, $W_{j,k}$ is computed using Gaussian quadrature with four points in each dimension, to avoid the evaluation of $W$ at the origin.
\end{example}

\begin{figure}[thp]
\centering
\subfloat[$\epsilon = 0$]{\includegraphics[totalheight=0.24\textheight]{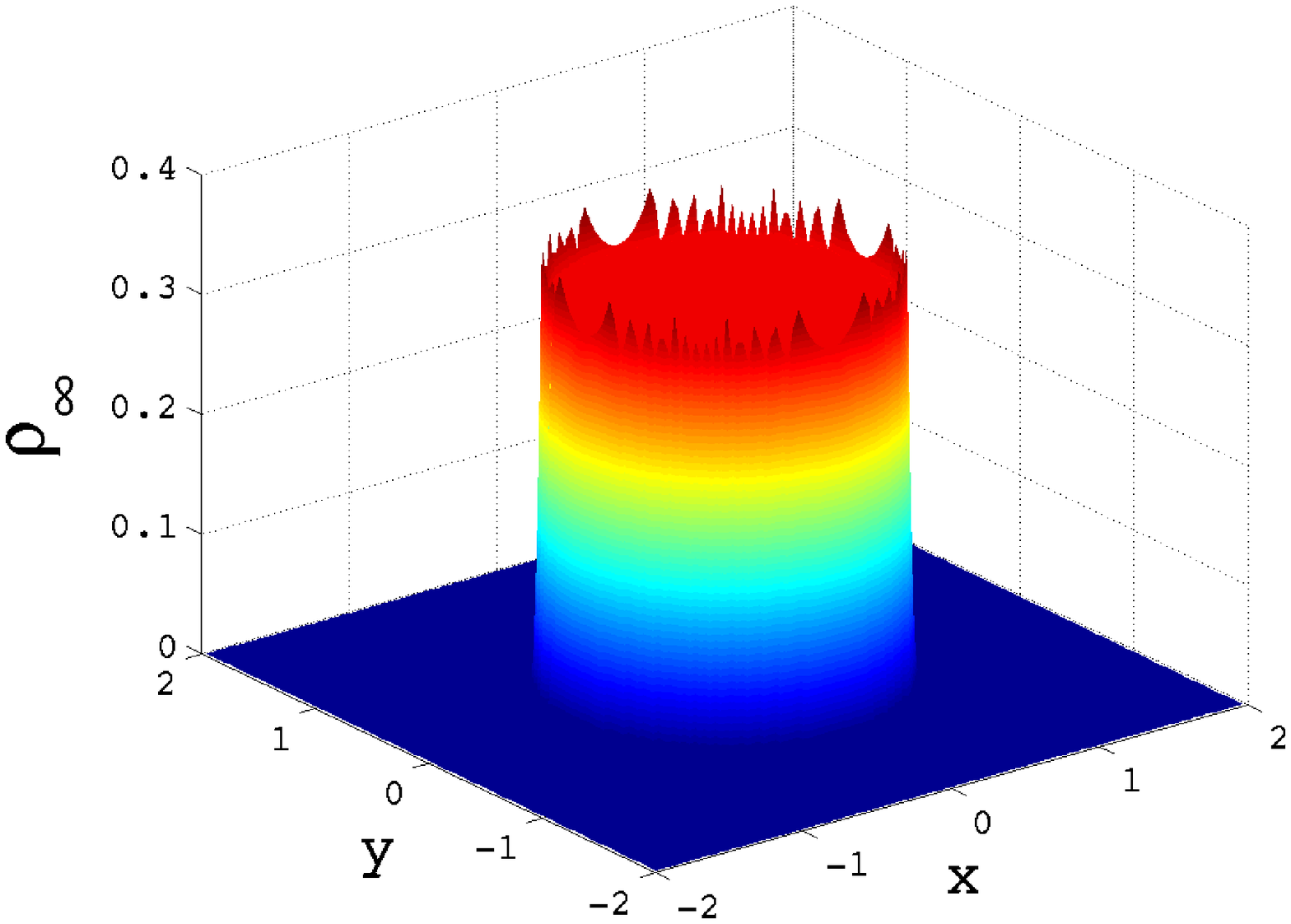}}
$~~$
\subfloat[$\epsilon = 0.4((\Delta x)^2+(\Delta y)^2)$]{\includegraphics[totalheight=0.24\textheight]{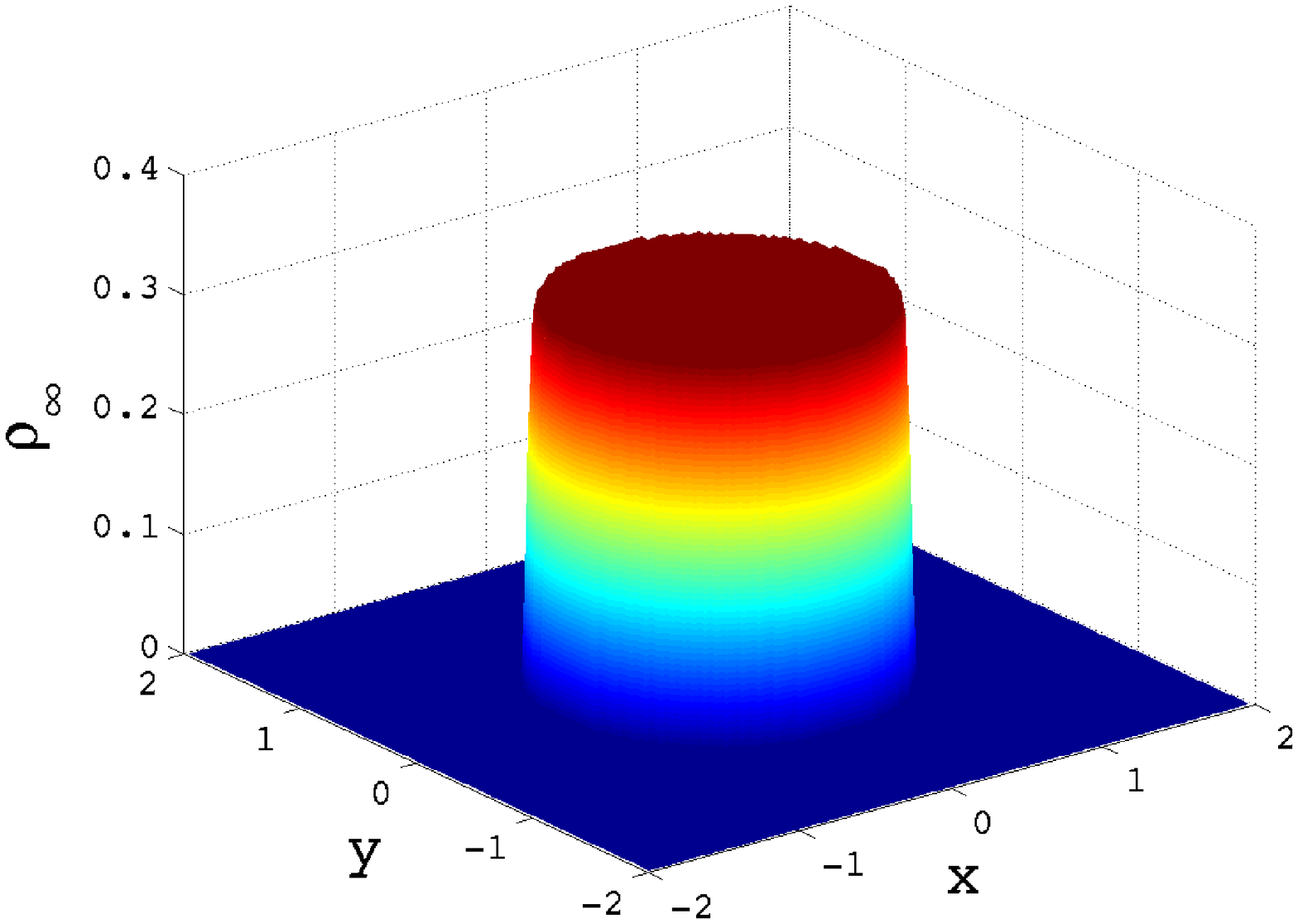}}
\caption{{\bf (a)} the steady state of the equation with $W(\mathbf{x}) =|\mathbf{x}|^2/2 - \ln |\mathbf{x}|$; {\bf (b)} the steady state with the same $W(\mathbf{x})$, regularized by quadratic diffusion $\nabla\cdot\big(\rho\nabla(\epsilon\rho)\big)$. The exact steady state 
without diffusion is the characteristic function of the unit disk with density $\frac{1}{\pi}$.}
\label{fig:quadlog2d}
\end{figure}
\end{example}

\begin{example}[{\bf Steady mill solutions}]\label{ex311}
Another common pattern observed for the self-propelled particle systems with an attractive-repulsive kernel in 2-D is the rotating mill~\cite{MR2507454}, and the steady 
pattern can be obtained from the equation
\[
\rho_t = \nabla \cdot \big(\rho \nabla(W\ast \rho-\frac{\alpha}{\beta} \log |\mathbf{x}|)\big), 
\quad \mathbf{x} \in \mathbb{R}^2,
\]
with some positive constants $\alpha$ and $\beta$.
For the kernel $W(\mathbf{x}) = \frac{1}{2}|\mathbf{x}|^2-\ln |\mathbf{x}|$, the steady state is still a constant $\rho_\infty=2$ on an annulus, whose inner and outer radius are
given by 
$$
R_0 = \sqrt{\frac{\alpha}{\beta}},\quad 
R_1 = \sqrt{\frac{\alpha}{\beta}+\frac{M}{2\pi}}, 
$$
with the total conserved mass $M=\int_{\mathbb{R}^d} \rho d{\bf x}$. For other more 
realistic kernels like the Morse type~\cite{MR2507454} or Quasi-Morse type~\cite{Carrillo2013112}, 
the radial density is in general more concentrated near the inner radius, but the 
explicit form of $\rho_\infty$ can not be obtained in general.  Numerical diffusion, in the form of $\epsilon \nabla \cdot(\rho \nabla \rho)$, is still needed 
to prevent the overshoot and the resulting steady states with $\epsilon=
0.2((\Delta x)^2+(\Delta y)^2)$ are shown in Figure~\ref{fig:mill} for 
two different potentials.

\begin{figure}[thp]
\centering
\subfloat[$W(\mathbf{x})=|\mathbf{x}|^2/2-\ln |\mathbf{x}|$]{\includegraphics[totalheight=0.29\textheight]{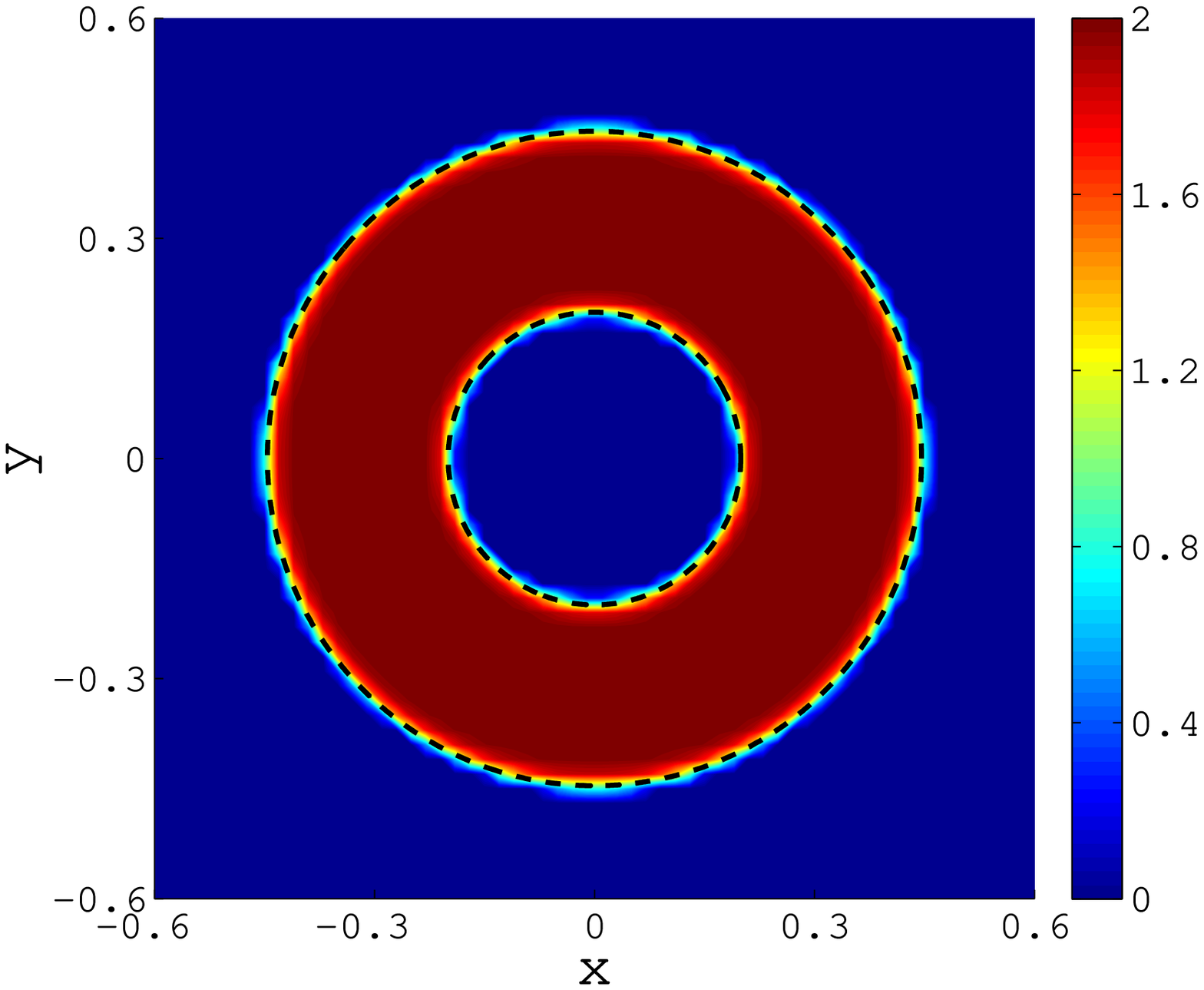}}
$~~$
\subfloat[$W(\mathbf{x})=\lambda\big(V(|\mathbf{x}|)-CV(|\mathbf{x}|/\ell)\big)$]{\includegraphics[totalheight=0.29\textheight]{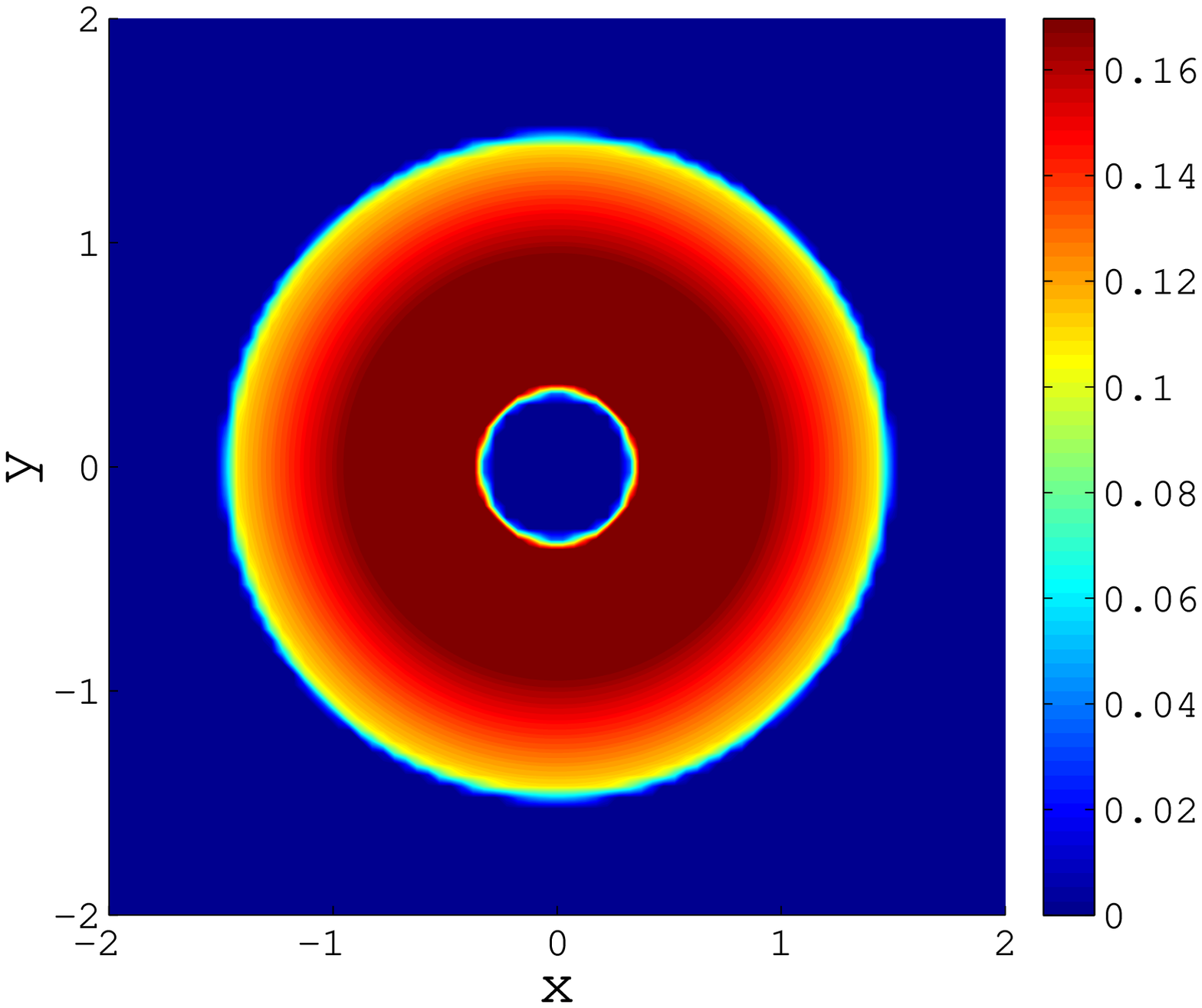}}
\caption{The steady density $\rho_\infty$ for the rotating mill with $\Delta x=\Delta y=0.05$. {\bf (a)} $\alpha = 0.25$, $\beta=2\pi$; {\bf (b)} $V(r) = -K_0(kr)/2\pi$, where $K_0(r)$ is the modified Bessel function of the second kind and the parameters $C = 10/9, \ell=0.75, k =0.5, \lambda=100, \alpha = 1.0, \beta=40$ are taken from\cite{Carrillo2013112}.}
\label{fig:mill}
\end{figure}
\end{example}


\begin{acknowledgments}
JAC acknowledges support from projects MTM2011-27739-C04-02,
2009-SGR-345 from Ag\`encia de Gesti\'o d'Ajuts Universitaris i de Recerca-Generalitat de Catalunya, and the Royal Society through a Wolfson Research Merit Award. JAC and YH were supported by Engineering and Physical Sciences Research Council (UK) grant number EP/K008404/1. The work of AC was supported in part by the NSF Grant DMS-1115682. The authors also acknowledge the support by NSF RNMS grant DMS-1107444. 
\end{acknowledgments}

\end{document}